\documentclass[12pt]{article}%
\usepackage{amsmath}
\usepackage{amsfonts}
\usepackage{amssymb}
\usepackage{graphicx}%
\providecommand{\U}[1]{\protect\rule{.1in}{.1in}}
\newtheorem{theorem}{Theorem}

\newtheorem{corollary}[theorem]{Corollary}

\newtheorem{lemma}[theorem]{Lemma}

\newtheorem{proposition}[theorem]{Proposition}

\newenvironment{proof}[1][Proof]{\noindent\textbf{#1.} }{\ \rule{0.5em}{0.5em}}
\begin{document}

\title{Hilbert and Fr\'{e}chet bundle versions of the Harish-Chandra and Whittaker
Plancherel Theorems}
\author{Nolan R. Wallach}
\maketitle

\begin{abstract}
This paper, in particular, gives a complete proof of the direct integral
version of the Whittaker Plancherel Theorem. The main emphasis is on certain
Hilbert and Fr\'{e}chet vector bundles over a space that has a submersion onto
the tempered dual. This allows for an approach to the Plancherel Theorems
(both for $L^{2}$ and the Whittaker case) that is representation theoretic and
bypasses the need for Harish-Chandra's Eisenstein Integrals and yields a proof
the direct integral decompositions without invoking the abstract theory.

\end{abstract}

\section{Introduction}

In \cite{WPT}, I proved an inversion formula for what Harish-Chandra called
the Whittaker Integral, but is more justifiably called the Jacquet Integral,
on a real reductive group as a consequence of the Harish-Chandra inversion
formula for what he called the Eisenstein Integral. In Chapter 14 my book
\cite{RRGII} I showed how Harish-Chandra's Theorem calculates the Plancherel
Measure. In Chapter 15 of \cite{RRGII} I stated a direct integral form of the
Whittaker Plancherel Theorem which was based on several results whose proofs
had serious gaps and whose statements are not quite correct (see the remark
below about $\chi$ and $\chi^{-1}$). In \cite{WPT} the proofs and statements
were fixed. In this paper it was shown, in particular, that the direct
integral formula, as stated, in \cite{RRGII} is correct. However, the main
thrust of this paper is to set up a formalism in the Harish-Chandra case that
is quite explicit and makes the necessary constructs more natural using
Hilbert (and Fr\'{e}chet) vector bundles. This naturality allows us to move to
the Whittaker case relatively smoothly. In Section \ref{Direct-integral} the
direct integral version of the Harish-Chandra and the Whittaker Plancherel
Theorems are shown to easily follow from the vector bundle version. The
methods used in this paper may seem a bit pedantic but they make certain
aspects of the intertwining operators that involve choosing elements in a
coset defining an element of a Weyl group more natural. Almost half of this
paper is devoted to establishing this naturality. In the usual presentation of
the intertwining operators between representations induced from the same
parabolic subgroup there is a choice of a representative in the coset that
defines the corresponding element of the Weyl group. In this paper the
intertwining operator is given in terms of any element of the corresponding
coset (see the discussion following Lemma \ref{Totality}).

Althouhgh this paper rests on the inversion formulae for the Harish-Chandra
and Whittaker transforms. The inversion formula does not immediately imply the
direct integral formula which asserts an isomorphism of unitary representations.

One mysterious aspect of \cite{WPT} (at least to me) was the appearance of
$\chi^{-1}$ rather than $\chi$ in the distribution version of the Whittaker
Plancherel formula. In section \ref{Direct-integral} this "mystery" is solved,
and has to do with that fact that multiplicity space of a tempered unitary
representation in the abstract Plancherel theorem is its conjugate representation.

\section{Notation}

This section is meant to set our notation. One can find more details of
assertions, given here without reference, in \cite{RRGI} chapter 2.

If ${}^{W}X_{Y}^{Z}$ is the designation of a real Lie Group then {}$^{W}%
$fraktur(lower-case(X))$_{Y}^{Z}$ will denote its Lie algebra. If any of the
$Y,Z,W$ are missing then leave them off of the corresponding Lie algebra. Thus
$Lie(G)=\mathfrak{g}$, $Lie\left(  N_{P}\right)  =\mathfrak{n}_{P}$,
$Lie({}^{o}M_{P})={}^{o}\mathfrak{m}_{P}$.

Let $G$ be a real reductive group with maximal compact subgroup, $K$ and let,
$\theta$, denote the Cartan involution of $G$ corresponding to $K$,that is,
the involutive automorphism of $G$ such that its fixed point set is $K$. Let
$B$ denote an $Ad(G)$--invariant symmetric bilinear form on $\mathfrak{g} $
such that
\[
\left\langle x,y\right\rangle =-B(\theta x,y)
\]
defines an inner product on $G$. Let $G=KA_{o}N_{o}$ be an Iwasawa
decomposition of $G$ which, as usual, means that $A_{o}$ is connected subgroup
of $G$ such that if $a\in A_{o}$ then $\theta(a)=a^{-1}$ and $N_{o} $ is a
connected simply connected nilpotent Lie subgroup of $G$ such that $Ad(N_{o})$
consists of unipotent elements, $\ker Ad_{|N_{o}}$ is trivial and
$aN_{o}a^{-1}\subset N_{o}$ for all $a\in A_{o}$ and the map%
\[
K\times A_{o}\times N_{o}\rightarrow G
\]
given by multiplication is a diffeomorphism. Throughout this paper, $K,A_{o}$
and $N_{o}$ will be fixed.

If $S\subset G$ and $H$ is a subgroup of $G$ then we set $N_{H}(S)=\{h\in
H|hSh^{-1}\subset S\}$ and $C_{H}(S)=\{h\in H|hsh^{-1}=s,s\in S\}$. Set
$M_{o}=C_{K}(A_{o})$ and $P_{o}=M_{o}A_{o}N_{o}$. We define a standard
parabolic subgroup of $G$ to mean a subgroup, $P$, that is its own normalizer
and such that $P\supset P_{o}$. Then, clearly $P_{o}$ is the minimal standard
parabolic subgroup of $G$. A parabolic subgroup of $G$ is a subgroup conjugate
under $G$ (hence conjugate under $K$) to a standard parabolic subgroup. This
is equivalent to the usual definition of parabolic subgroup (as subgroup that
is its own normalizer and whose complexified Lie algebra contains a Borel
subalgebra), since all Iwasawa decompositions are conjugate under $K$ and all
maximal compact subgroups are conjugate under $G$.

If $H$ is a real reductive group with maximal compact subgroup $U$ and a
corresponding Cartan involution, $\eta$, let $X(H)$ denote the set of
continuous homomorphisms of $H$ to the positive real numbers (under
multiplication). Let
\[
{}^{o}H=\cap_{\eta\in X(H)}\ker\eta.
\]
If $C$ is the center of $H$ then the identity component, $C_{o}$, contains a
unique maximal compact torus which is the identity component of{}$^{o}C=C\cap
U$. Set $S_{\eta,H}=\{x\in C_{o}|\eta(x)=x^{-1}\}$ then $S_{\eta,H}$ is a
subgroup of $C_{o}$ that is isomorphic with $\mathbb{R}^{m} $ for some $m$.

If $P$ is a parabolic subgroup of $G$ then set $M_{P}=\theta(P)\cap P$.
$M_{P}$ is a real reductive group and we define $A_{P}=S_{\theta_{|M_{P}%
},M_{P}}$. Then $M_{P}$ is a Levi factor of $P$. Any Levi factor of $P$ is
conjugate to $M_{P}$ in $P$. Let $N_{P}$ be the unipotent radical of $Ad(P)$.
Then the map%
\[
N_{P}\times A_{P}\times{}^{o}M_{P}\rightarrow P
\]
given by multiplication is a diffeomorphism, known as the Langlands
decomposition of $P$.

If $P$ is a parabolic subgroup with Langlands decomposition $P=N_{P}A_{P}%
{}^{o}M_{P}$ then $(P,A_{P})$ is called a parabolic pair or a $P$--pair. If
$Q=kPk^{-1}$ with $P$ a standard parabolic subgroup then, clearly,
$\ ^{o}M_{Q}=k{}^{o}M_{P}k^{-1}.$ The unipotent radical of $Ad(Q)$ is
$Ad(k^{-1}N_{P}k)$ thus we set $N_{Q}=k^{-1}N_{P}k$ and if $A_{Q}=kA_{P}%
k^{-1}$ then $M_{Q}=A_{Q}\ ^{o}M_{Q}$. We define, for $(P,A_{P})$ a standard
$P$--pair, $\bar{P}=\theta(P)$. Then $\bar{P}$ is a parabolic subgroup of $G$
but not necessarily equal to $kPk^{-1}$ for some $k\in K$. In any event, using
our definitions, we have $M_{\overline{P}}=M_{P},A_{\overline{P}}=A_{P}$ and
$N_{\overline{P}}=\theta N_{P}$. For simplicity we will label $\theta
N_{P}=\bar{N}_{P\text{ }}$ and its Lie algebra $\mathfrak{\bar{n}}_{P}$.

Parabolic subgroups, $P$ and $Q,$ are said to be associate if $M_{P}$ and
$M_{Q}$ are conjugate in $G$ (if $P$ and $Q$ are standard parabolic subgroups
then $P$ and $Q$ are associate if and only if $A_{P},A_{Q}\subset A_{o}$ are
conjugate in the normalizer of $A_{o}$ in $K$). Let $\mathcal{P}(G)$ be the
set of associativity classes of parabolic subgroups of $G$ every associativity
class contains a standard parabolic subgroup. If $A\subset A_{o}$ is a
subgroup then we will call it a standard split component if there exists a
standard parabolic subgroup, $P$, such that $A=A_{P}$.

A parabolic subgroup of $G$, $P$, is said to be cuspidal if ${}^{o}M_{P}$ has
a compact Cartan subgroup. If $P$ is cuspidal and $T$ is a compact Cartan
subgroup of ${}^{o}M_{P}$ then $H=TA_{P}$ is a Cartan subgroup of $G$ which we
will call associated with $P$. Any pair of Cartan subgroups associated with
$P$ are conjugate by an element of $M_{P}$. The following result is well known
we give a proof since it is hard to find explicitly stated (for example it is
implicit in \cite{HCI}).

\begin{proposition}
The following are equivalent for cuspidal parabolic subgroups $P$ and $Q$.

1. $P$ and $Q$ are associate.

2. The Lie algebra of a compact Cartan subgroup of $M_{P}$ is conjugate to a
Lie algebra of compact subgroup of $M_{Q}$

3. $A_{P}$ and $A_{Q}$ are conjugate relative to $G$.

4. The Cartan subgroups associated with $P$ are all conjugate relative to $G$
to the ones associated with $Q$.

Furthermore, if $H$ is a Cartan subgroup of $G$ then $H$ is associated with a
Cuspidal parabolic subgroup.
\end{proposition}

\begin{proof}
We prove 1. implies 2. implies 3. implies 1. and 3. implies 4 implies 2.

1. implies that there exists $g\in G$ such that $g{}^{o}M_{P}g^{-1}={}%
^{o}M_{Q}$. Since all compact Cartan subgroups in ${}^{o}M_{Q}$ are conjugate
in ${}^{o}M_{Q}$ it is now clear that 1. implies 2. Note that its is also
clear that 3. implies 1 since $M_{P}=C_{G}(A_{P})$.

Let $\mathfrak{t}$ be the Lie algebra if a compact Cartan subgroup, $T$, of
${}^{o}M_{P}$ and let $g$ be such that $gTg^{-1}$ is a compact Cartan subgroup
of ${}^{o}M_{Q}$ set $Ad(g)\mathfrak{t=t}_{1}$. Let $\mathfrak{l}%
=C_{\mathfrak{g}}(\mathfrak{t})$. Then $\mathfrak{a}_{P}$ is a Cartan
subalgebra of $\mathfrak{l}$. Indeed of $\mathfrak{b}$ is a Cartan subalgebra
of $\mathfrak{l}$ containing $\mathfrak{a}_{P}$ then $\mathfrak{t}%
\oplus\mathfrak{b}$ is an abelian subalgebra of $\mathfrak{g}$ consisting of
semi-simple elements. Clearly $\mathfrak{b}\subset\mathfrak{m}_{P}$ so
$\mathfrak{b}=\mathfrak{a}_{P}\oplus\mathfrak{w}$ with $\mathfrak{w}\subset
{}^{o}\mathfrak{m}_{P}$. Since $\mathfrak{t}$ is a Cartan subalgebra of
${}^{o}\mathfrak{m}_{P}$ this implies that $\mathfrak{w=}0$. Arguing in the
same way we see that $\mathfrak{a}_{Q}$ is a Cartan subalgebra of
$Ad(g)\mathfrak{l}$. Since all Cartan subalgebras that are split over
$\mathbb{R}$ of Chevalley group over $\mathbb{R}$ are conjugate, this proves
that 2. implies 3. So at this point we have shown that 1.,2.,3. are equivalent.

Since $M_{P}=C_{G}(A_{P})$ and $M_{Q}=C_{G}(A_{Q})$, and all compact subgroups
of ${}^{o}M_{P}$ (resp. ${}^{o}M_{Q}$) are conjugate relative to ${}^{o}M_{P}$
(resp. ${}^{o}M_{Q}$) we see that 3. implies 4.

If $H$ is an abelian real reductive group then ${}^{o}H$ is the unique maximal
compact subgroup of $H$ thus 4. implies 2.

The last assertion is an easy consequence of the well known assertion: If
$\mathfrak{h}$ is a Cartan subalgebra of $\mathfrak{g}$ then $\mathfrak{h}$
has a $\theta$--stable $G$ conjugate (c.f. \cite{RRGI} Lemma 2.3.2 with the
typo on line -11 page 56 $dq_{0.Z}$ should say $dq_{0,X}$) .
\end{proof}

If $(P,A)$ is a parabolic pair in $G$ then set $\Phi(P,A)$ equal to the
non-zero weights of $\mathfrak{a}_{P}$ acting on $\mathfrak{n}_{P}$ under the
restriction of the adjoint representation.

Let $(\pi,H)$ is a strongly continuous representation of $G$ on a Fr\'{e}chet
space. Then $H^{\infty}$ will denote the $C^{\infty}$ vectors. $(\pi,H)$ is
said to be smooth if for each $v\in V$ the map $g\mapsto\pi(g)v$ is
$C^{\infty}$. Then $(\pi,H^{\infty})$ is a smooth Fr\'{e}chet representation
(here the Fr\'{e}chet structure is given by adding the semi-norms $q\circ
d\pi(u)$ for $q$ a continuous seminorm on $H$ and $u\in$ $U(\mathfrak{g})$).

Let $(P,A_{P})$ be a parabolic pair and let $(\sigma,H_{\sigma})$ be a
strongly continuous representation of ${}^{o}M_{P}$ on a Hilbert space such
that $\sigma_{|K\cap M_{P}}$ is unitary. If $\nu\in\left(  \mathfrak{a}%
_{P}\right)  _{\mathbb{C}}^{\ast}$ then define the representation $\sigma
_{\nu}$ of $M_{P}$ on $H_{\sigma}$ by%
\[
\sigma_{\nu}(am)=a^{\nu+\rho_{P}}\sigma(m)
\]
with $a\in A_{P,}m\in{}^{o}M_{P}$ and $\rho_{P}(h)=\frac{1}{2}\mathrm{tr}%
(ad(h)_{|\mathfrak{n}_{P}})$ for $h\in\mathfrak{a}_{P}$. Let $H_{[P],\sigma
}^{\infty}$ be the space of $C^{\infty}$ maps, $f$, of $K$ into $H_{\sigma
}^{\infty}$ such that $f(mk)=\sigma(m)f(k)$ for $m\in$ $M_{P}$ $\cap K$ and
$k\in K$. Note that $H_{[P],\sigma}^{\infty}$ depends only on $\sigma
_{|M_{P}\cap K}$ that is the reason for the the designation $[P]$. If $f,h\in
H_{[P],\sigma}^{\infty}$ then the inner product is given by%
\[
\left\langle f,h\right\rangle =\int_{K}\left\langle f(k),h(k)\right\rangle dk
\]
and take $H_{\left[  P\right]  ,\sigma}$ equal to the Hilbert space completion
of $H_{[P],\sigma}^{\infty}$ relative to this inner product. Note that
$H_{[P],\sigma}$ is just the unitarily induced representation from $M_{P}\cap
K$ to $K$ of $\sigma_{|M_{P}\cap K}$. For each $\nu\in\left(  \mathfrak{a}%
_{P}\right)  _{\mathbb{C}}^{\ast}$ we define a strongly continuous
representation of $G$ on $H_{[P],\sigma}$ as follows: If $f\in H_{\left[
P\right]  ,\sigma}$ then for $g=nmk$ with $n\in N_{P},m\in M_{P}$ and $k\in K$
set
\[
{}_{P}f_{\nu}(nmk)=\sigma_{\nu}(m)f(k).
\]
Then
\[
\left(  \pi_{P,\sigma,\nu}(g)f\right)  (k)={}_{P}f_{\nu}(kg).
\]
The $C^{\infty}$ vectors of $(\pi_{P,\sigma,\nu},H_{\left[  P\right]  ,\sigma
})$ are the $f\in H_{\left[  P\right]  ,\sigma}^{\infty}$ with the usual
Fr\'{e}chet topology of uniform convergence with all derivatives. Note that
the space of $C^{\infty}$ vectors for $\pi_{P,\sigma,\nu}$ is independent of
$\nu$.

Let $\hat{K}$ denote the set of equivalence classes of irreducible continuous
representations of $K$. If $\gamma\in\hat{K}$ then fix a representative
$(\tau_{\gamma},V_{\gamma})$ of $\gamma$. If $(\pi,H)$ is a a strongly
continuous representation of $K$ on a Fr\'{e}chet space and if $\gamma\in
\hat{K}$ then set
\[
H[\gamma]=Span_{\mathbb{C}}\{TV_{\gamma}|T\in\mathrm{Hom}_{K}(V_{\gamma},H)\}
\]
Here $Hom_{K}$ means continuous $K$ intertwining operators. Then the space of
$K$--finite vectors
\[
H_{K}=\sum_{\gamma\in\hat{K}}H[\gamma]=\bigoplus_{\gamma\in\hat{K}}H[\gamma]
\]
is a dense subspace, $H[\gamma]$ is called the $\gamma$ isotypic component of
$H$. $(\pi,H)$ is said to be admissible if
\[
\dim H[\gamma]<\infty
\]
for all $\gamma\in\hat{K}$. If $\chi_{\gamma}$ is the character of $\gamma
\in\hat{K}$ and if $d_{\gamma}=\dim V_{\gamma}$ then set%
\[
E_{\gamma}v=d_{\gamma}\int_{K}\chi_{\gamma}(k^{-1})\pi(k)v
\]
for $v\in H$. Then $H[\gamma]=E_{\gamma}H$ and if $(\pi,H)$ is a Hilbert
representation then $E_{\gamma}$ is the orthogonal projection onto $H[\gamma
]$. If $F\subset\hat{K}$ is finite set $E_{F}=\sum_{\gamma\in F}E_{\gamma}$.

Let $d=\dim N_{o}$ and let $v_{o}$ be a unit vector in $\wedge^{d}%
\mathfrak{n}_{o}$ with inner product induced by $\left\langle
...,...\right\rangle $ and let $V_{o}$ be the linear span of $\wedge
^{d}Ad(G)v_{o}$. Set for $g\in G,$ $\left\vert g\right\vert $ equal to the
operator norm of $\wedge^{d}Ad(g)_{|V_{o}}$ and define for $g=(\exp h)g_{o}$
with $h\in\mathfrak{a}_{G}$
\[
\left\Vert g\right\Vert =e^{\sqrt{\left\langle h,h\right\rangle }}\left\vert
g\right\vert ^{\frac{1}{2}}%
\]
Then $\left\Vert ...\right\Vert $ is a norm on $G$ in the sense of \cite{RRGI}
2.A.2 and if $g\in{}^{o}G$ and $g=k_{1}ak_{2}$ with $k_{1},k_{2}\in K $ and
$a\in A_{o}$ such that all of the eigenvalues of $Ad(a)_{|\mathfrak{n}_{o}}$
are greater than or equal to $1$ then
\[
\left\vert g\right\vert ^{\frac{1}{2}}=\left\Vert g\right\Vert =a^{\rho_{o}}.
\]

We now recall the definition of the Harish-Chandra Schwartz Space. Let $f\in
C^{\infty}(G)$. If $x,y\in U(\mathfrak{g})$ (the universal enveloping algebra
of $\mathfrak{g})$ and $k\in\mathbb{R}$ define
\[
p_{x,y,k}(f)=\mathrm{sup}_{G}\left\vert L_{x}R_{y}f(g)\right\vert
(1+\log\left\Vert g\right\Vert )^{k}\left\vert g\right\vert ^{\frac{1}{2}}.
\]
here $L$ and $R$ are respectively the left and right regular actions of $G$ on
$f\in C^{\infty}(G)$. Then $\mathcal{C}(G)$ denotes the space of $f\in
C^{\infty}(G)$ such that $p_{x,y,k}(f)<\infty$ for all $k,x,y$. $\mathcal{C}%
(G)$ is endowed with the topology induced by the seminorms $p_{x,y,k}$ and is
a Fr\'{e}chet Space such that $L$ and $R$ define smooth representations of
$G$. In particular, we may define $E_{\gamma,L}$ and $E_{\gamma,R}$ on
$\mathcal{C}(G)$ for the $E_{\gamma}$ for respectively the right and left
regular representation). Also, in Harish-Chandra's theory a specific function
plays a key role%
\[
\Xi(g)=\left\langle \pi_{P_{o},1,0}(g)1,1\right\rangle
\]
and there exist $C,d>0$ \ such that (c.f. Theorem 4.5.3 \cite{RRGI})%
\[
\overset{}{(\ast)}\left\vert g\right\vert ^{-\frac{1}{2}}\leq\Xi(g)\leq
C\left\vert g\right\vert ^{\frac{1}{2}}(1+\log|g|)^{d}.
\]

So our definition is the same as that of Harish-Chandra.

A finitely generated, admissible, representation,$(\pi,H)$, of $G$ is said to
be tempered if for each $v,w\in H^{\infty}$ the function $c_{v,w}(g)$
satisfies the weak inequality. That is, there exist $k,C\in\mathbb{R}_{>0}$
such that%
\[
\left\vert c_{v,w}(g)\right\vert \leq C(1+\log\left\Vert g\right\Vert
)^{k}\left\vert g\right\vert ^{-\frac{1}{2}}.
\]
This is usually stated with $v,w\in H_{K}$ and $\left\vert g\right\vert
^{-\frac{1}{2}}$ replaced by $\Xi(g)$, but it is true in this generality, see
\cite{RRGII} Theorem 15.2.4).

Let $C_{K}$ denote the Casimir operator on $K$ corresponding to
$B_{|\mathfrak{k}}$. Let denote $\lambda_{\gamma}$ be the eigenvalue of
$C_{K}$ on any representative of $\gamma\in\hat{K}$. Note that since $C_{K}$
is a second order elliptic operator one has%
\[
\sum_{\gamma\in\hat{K}}\frac{d_{\gamma}}{(1+\lambda_{\gamma})^{\frac{\dim
K}{2}+\varepsilon}}<\infty
\]
for all $\varepsilon>0$.

\section{The Plancherel Space of $G\label{G-Plancherel Space}$}

Using the subquotient theorem (c.f. \cite{RRGI} Theorem 3.5.6), induction in
stages, the Weyl dimension formula and the formula for $\lambda_{\gamma}$ in
terms of the highest weight of $\gamma$ we have

\begin{proposition}
Let $(P,A)$ be a parabolic pair and let $(\sigma,H_{\sigma})$ be an
irreducible representation of ${}^{o}M_{P}$ and $\nu\in\mathfrak{a}%
_{\mathbb{C}}^{\ast}$ then
\[
\dim H_{[P],\sigma}[\gamma]\leq d_{\gamma}^{2}\leq C(1+\lambda_{\gamma})^{\dim
K-rank(K)}.
\]

\end{proposition}

\begin{proposition}
\label{tempered-CK}Let $(P,A)$ be a parabolic pair and $(\sigma,H_{\sigma}) $
be an irreducible square integrable representation of ${}^{o}M_{P}$ then there
exist $r_{\sigma}$ and $C_{\sigma}$ such that if $\gamma.\eta\in\hat{K} $ and
$u\in H_{[P],\sigma}[\gamma]$ and $w\in H_{[P],\sigma}[\eta] $ then
\[
\left\vert \left\langle \pi_{P,\sigma,i\nu}(g)u,w\right\rangle \right\vert
\leq C_{\sigma}(1+\lambda_{\gamma})^{r_{\sigma}}(1+\lambda_{\eta})^{r_{\sigma
}}\left\Vert u\right\Vert \left\Vert w\right\Vert \Xi(g)
\]%
\[
\leq C_{\sigma}C(1+\lambda_{\gamma})^{r_{\sigma}}(1+\lambda_{\eta}%
)^{r_{\sigma}}\left\Vert u\right\Vert \left\Vert w\right\Vert \left\vert
g\right\vert ^{-\frac{1}{2}}(1+\log\left\Vert g\right\Vert )^{d}%
\]
with $C,d$ as in $(\ast)$ in the previous section.
\end{proposition}

\begin{proof}
In \cite{WPT} Corollary 65 we proved that there exists a continuous seminorm
$q$ on $H_{\left[  P\right]  ,\sigma}^{\infty}$ such that if $f,h\in
H_{[P],\sigma}^{\infty}$ then%
\[
\left\vert \left\langle \pi_{P,\sigma,i\nu}(g)f,h\right\rangle \right\vert
\leq q(f)q(h)\Xi(g).
\]
Since the $K$--$C^{\infty}$ vectors of $\pi_{P,\sigma,i\nu}$ are the same as
the $G$--$C^{\infty}$ vectors (\cite{CWParam} Proposition A.2) there exists
$d_{\sigma}$ and $B_{\sigma}$ such that
\[
q(f)\leq B_{\sigma}\left\Vert (1+C_{K})^{d_{\sigma}}f\right\Vert .
\]
Take $C_{\sigma}=B_{\sigma}^{2}$.
\end{proof}

Let $\{v_{i}\}$ be an orthonormal basis of $H_{[P],\sigma}$ such that for each
$i$ there exists $\gamma_{i}\in\hat{K}$ such that $v_{i}\in H_{[P],\sigma
}[\gamma_{i}]$.

\begin{proposition}
\label{Weyl-ests}Let $f\in\mathcal{C}(G)$ then for each $r,s$ there exists a
continuous seminorm $q_{r,s,\sigma}$ such that%
\[
\left\vert \int_{G}f(g)\left\langle \pi_{P,\sigma,i\nu}(g)v_{i},v_{j}%
\right\rangle dg\right\vert \leq\frac{q_{r,s,\sigma}(f)}{(1+\lambda
_{\gamma_{i}})^{r}(1+\lambda_{\gamma_{j}})^{s}}.
\]

\end{proposition}

\begin{proof}
Note that for all $l>0$ we have%
\[
\left\vert \left(  L_{(1+C_{K})^{l}}R_{(1+C_{K})^{m}}f\right)  (g)\right\vert
\leq p_{(1+C_{K})^{l},(1+C_{K})m,d}(f)\Xi(g)(1+\log\left\Vert g\right\Vert
)^{-l}.
\]
Thus for all $r$%
\[
\left\vert \int_{G}\left(  L_{(1+C_{K})^{l}}R_{(1+C_{K})^{m}}f\right)
(g)\left\langle \pi_{P,\sigma,i\nu}(g)v_{i},v_{j}\right\rangle dg\right\vert
\leq
\]%
\[
C_{\sigma}p_{(1+C_{K})^{l},(1+C_{K})^{m},d+r}(f)C_{\sigma}(1+\lambda
_{\gamma_{i}})^{d_{\sigma}}(1+\lambda_{\gamma_{j}})^{d_{\sigma}}\int
_{G}\left\vert g\right\vert ^{-1}(1+\log\left\Vert g\right\Vert )^{-r}dg.
\]
($d$ is as in Proposition \ref{tempered-CK}$)$Take $r>0$ so large that
\[
\int_{G}\left\vert g\right\vert ^{-1}(1+\log\left\Vert g\right\Vert
)^{-r}dg<M<\infty.
\]
We calculate%
\[
\int_{G}\left(  L_{(1+C_{K})^{l}}R_{(1+C_{K})^{m}}f\right)  f(g)\left\langle
\pi_{P,\sigma,i\nu}(g)v_{i},v_{j}\right\rangle dg=
\]%
\[
\int_{G}f(g)\left\langle \pi_{P,\sigma,i\nu}(1+C_{K})^{l}\pi_{P,\sigma,i\nu
}(g)\pi_{P,\sigma,i\nu}(1+C_{K})^{m}v_{i},v_{j}\right\rangle dg=
\]%
\[
\int_{G}f(g)\left\langle \pi_{P,\sigma,i\nu}(g)\pi_{P,\sigma,i\nu}%
(1+C_{K})^{m}v_{i},\pi_{P,\sigma,i\nu}(1+C_{K})^{l}v_{j}\right\rangle =
\]%
\[
(1+\lambda_{\gamma_{i}})^{m}(1+\lambda_{\gamma_{j}})^{l}\int_{G}%
f(g)\left\langle \pi_{P,\sigma,i\nu}(g)v_{i},v_{j}\right\rangle dg.
\]
Thus%
\[
\left\vert \int_{G}f(g)\left\langle \pi_{P,\sigma,i\nu}(g)v_{i},v_{j}%
\right\rangle dg\right\vert \leq
\]%
\[
(1+\lambda_{\gamma_{i}})^{-m}(1+\lambda_{\gamma_{j}})^{-l}\left\vert \int
_{G}\left(  L_{(1+C_{K})^{l}}R_{(1+C_{K})^{m}}f\right)  (g)\left\langle
\pi_{P,\sigma,i\nu}(g)v_{i},v_{j}\right\rangle dg\right\vert
\]%
\[
\leq C_{\sigma}p_{(1+C_{K})^{l},(1+C_{K})m,d+r}(f)(1+\lambda_{\gamma_{i}%
})^{r_{\sigma}-m}(1+\lambda_{\gamma_{j}})^{r_{\sigma}-l}.
\]
Take $q_{r,s,\sigma}(f)=C_{\sigma}p_{(1+C_{K})^{l},(1+C_{K})m,d+r}.$
\end{proof}

This result implies that if we define for $\nu\in\mathfrak{a}^{\ast}$%
\[
a_{ij}(\nu)(f)=\int_{G}f(g)\left\langle \pi_{P,\sigma,i\nu}(g)v_{i}%
,v_{j}\right\rangle dg
\]
then
\[
\sum_{ij}\left\vert a_{ij}(\nu)(f)\right\vert \leq\sum_{ij}\frac
{q_{r,s,\sigma}(f)}{(1+\lambda_{\gamma_{i}})^{r}(1+\lambda_{\gamma_{j}})^{s}}%
\]
for all $r.s>0$. Thus if $r,s>\frac{\dim K}{2}$ then the series converges.
Thus, defining
\[
A(\nu)(f)v_{i}=\sum a_{ij}(\nu)(f)v_{j}%
\]
then $A(\nu)$ defines a trace class operator on $H_{[P],\sigma}$ and that
\[
\left\vert \mathrm{tr}A(\nu)(f)\right\vert \leq Cq_{r,s,\sigma}(f)
\]
also note that $A(\nu)(f)$ is of Hilbert-Schmidt class and%
\[
\sqrt{\mathrm{tr}A(\nu)(f)A(\nu)(f)^{\ast}}\leq C^{\prime}q_{r,s,\sigma}(f).
\]
Define $\pi_{P,\sigma,i\nu}(f)=A(\nu)(f)$ and note that using integration with
values in Fr\'{e}chet spaces one has%
\[
\pi_{P,\sigma,i\nu}(f)v=\int_{G}f(g)\pi_{P,\sigma,i\nu}(g)vdg
\]
for $v\in H_{[P],\sigma}^{\infty}$

Let $\mathcal{S(}H_{[P],\sigma})$ be the set of all bounded operators, $T$, on
$H_{[P],\sigma}$ such that $\pi_{P,\sigma}(1+C_{K})^{l}T\pi_{P,\sigma}%
(1+C_{K})^{m}$ is a bounded operator for all $l,k$. In terms of the
orthonormal basis that we have been using this means that for each $l,m>0$
there exists $C_{l,m}$ such that
\[
|\left\langle Tv_{i},v_{j}\right\rangle |\leq C_{l,m}(1+\lambda_{\gamma_{i}%
})^{-l}(1+\lambda_{\gamma_{j}})^{-m}.
\]
Endow $\mathcal{S(}H_{[P],\sigma})$ with the Fr\'{e}chet space topology given
by the seminorms $r_{l,k}(T)=\left\Vert \pi_{P,\sigma}(1+C_{K})^{l}%
T\pi_{P,\sigma}(1+C_{K})^{m}\right\Vert _{HS}$. We note that we could have
used the operator norm to get an equivalent set of seminorms.

Let $\mathcal{S}(G)$ be the space defined in 7.1.2 of \cite{RRGI}, that is,
the space of all $f\in C^{\infty}(G)$ such that%
\[
\mu_{r,x}(f)=\mathrm{sup}_{g\in G}\left\Vert g\right\Vert ^{r}\left\vert
xf(g)\right\vert <\infty
\]
for all $x\in U(\mathfrak{g})$ and $r>0$. If $(\pi,H)$ is a Banach
representation of $G$ and $f\in\mathcal{S}(G)$ then $\pi(f)H\subset H^{\infty
}$. One can show (\cite{RRGII} Theorem 11.8.2) that $\pi_{P,\sigma,i\nu}$ for
$\nu\in\mathfrak{a}^{\ast}$ is irreducible if and only if $\pi_{P,\sigma,i\nu
}(\mathcal{S}(G))=\mathcal{S(}H_{[P],\sigma})$. Since $\mathcal{C}%
(G)\supset\mathcal{S}(G)$ and by the above $\pi_{P,\sigma,i\nu}(\mathcal{C}%
(G))\subset\mathcal{S(}H_{[P],\sigma})$ , we have

\begin{lemma}
\label{Totality}If $\nu\in\mathfrak{a}^{\ast}$ and $\pi_{P,\sigma,i\nu} $ is
irreducible then $\pi_{P,\sigma,i\nu}(\mathcal{C}(G))=\mathcal{S(}%
H_{[P],\sigma})$.
\end{lemma}

Fix a parabolic pair $(P,A)$. Recall that $N_{K}(A)$ is the set of elements in
$k\in K$ such that $kAk^{-1}\subset A$. Let $(\sigma,H_{\sigma})$ be a
strongly continuous representation of ${}^{o}M_{P}$ on a Fr\'{e}chet Space. if
$k\in N_{K}(A)$ then define the representation $(k\sigma,H_{\sigma})$ by
$k\sigma(m)=(\sigma(k^{-1}mk)$ for $m\in{}^{o}M_{P}$, if $\nu\in
(\mathfrak{a}_{P})_{\mathbb{C}}^{\ast}$ then set $k\nu(h)=\nu(Ad(k)^{-1}h)$.
If $\nu\in(\mathfrak{a}_{P})_{\mathbb{C}}^{\ast}$ and $(Q,A)$ is another
parabolic pair (that is, $Q$ and $P$ are associate) let
\[
J_{Q|P}(\sigma,\nu):H_{[P],\sigma}^{\infty}\rightarrow H_{[Q],\sigma}^{\infty}%
\]
be as in 10.1.11 in \cite{RRGII}. Then we saw that $\nu\mapsto J_{Q|P}%
(\sigma,\nu)$, initially defined in an open subset of $(\mathfrak{a}%
_{P})_{\mathbb{C}}^{\ast}$ has a weakly meromorphic continuation to all of
$(\mathfrak{a}_{P})_{\mathbb{C}}^{\ast}$ of operators on $H_{[P],\sigma
}^{\infty}$satisfying%
\[
J_{Q|P}(\sigma,\nu)\pi_{P,\sigma,\nu}(g)=\pi_{Q,\sigma,\nu}(g)J_{Q|P}%
(\sigma,\nu),g\in G.
\]
Also note that if $k\in N_{K}(A)$ and if $f\in H_{[P],\sigma}^{\infty}$ then
$L_{k}(f)\in H_{P,k\sigma}^{\infty}$ and
\[
L_{k}\pi_{k^{-1}Qk,\sigma,\nu}(g)=\pi_{Q,k\sigma,k\nu}(g)L_{k}.
\]
Thus, if we define%
\[
A_{k}(P,\sigma,\nu)=L_{k}J_{k^{-1}Pk|P}(\sigma,\nu)
\]
then
\[
A_{k}(P,\sigma,\nu)\pi_{P,\sigma,\nu}(g)=\pi_{P,k\sigma,k\nu}(g)A_{k}%
(P,\sigma,\nu).
\]

Let $(\sigma,H_{\sigma})$ be an irreducible square integrable representation
of ${}^{o}M_{P}$. Let $\mu(\sigma,\nu)$ be Harish-Chandra's $\mu$--function
(with the additional constants for the Plancherel Theorem) and set
\[
\left(  \mathfrak{a}_{P}^{\ast}\right)  ^{\prime}=\{\nu\in\mathfrak{a}%
_{P}^{\ast}|(\alpha,\nu)\neq0,\alpha\in\Phi(P,A_{P})\}
\]
then we will be using the following result of Harish-Chandra

\begin{theorem}
If $\nu\in\left(  \mathfrak{a}_{P}^{\ast}\right)  ^{\prime}$ then $\mu
(\sigma,\nu)\neq0$.
\end{theorem}

\begin{proof}
This is a direct consequence of c.f. Lemma 13.6.4 and Theorem 10.5.7 (iii) in
\cite{RRGII}.
\end{proof}

This result combined with the Harish-Chandra determination of the intertwining
algebra of $\pi_{P,\sigma,i\nu}$ (c.f. Theorem 13.6.6 \cite{RRGII}) implies

\begin{theorem}
If $\nu\in\left(  \mathfrak{a}_{P}^{\ast}\right)  ^{\prime}$ and $k\in
N_{K}(A_{P})$ then $A_{k}(P,\sigma,i\nu)$ is defined and invertible as an
operator from $H_{[P],\sigma}^{\infty}$.to $H_{[P],k\sigma}^{\infty}$
\end{theorem}

If $\nu\in\left(  \mathfrak{a}_{P}^{\ast}\right)  ^{\prime}$ then use the
notation
\[
C_{k}(P,\sigma,i\nu):\mathcal{S}(H_{[P],\sigma})\rightarrow\mathcal{S}%
(H_{P,k\sigma})
\]
for the operator given by the formula
\[
C_{k}(P,\sigma,i\nu)(T)=A_{k}(P,\sigma,i\nu)TA_{k}(P,\sigma,i\nu)^{-1}%
\]
where defined (in particular for $\nu\in\left(  \mathfrak{a}_{P}^{\ast
}\right)  ^{\prime}$). If $H$ is a Hilbert space then we denote by $HS(H)$ the
Hilbert space of Hilbert-Schmidt operators on $H$ with the Hilbert-Schmidt
inner product.

\begin{lemma}
\label{extension}Let $HS(H_{[P],\sigma})$ denote the Hilbert-Schmidt operators
on $H_{[P].\sigma}$ with the Hilbert-Schmidt inner product. If $\nu\in\left(
\mathfrak{a}_{P}^{\ast}\right)  ^{\prime}$, then $C_{k}(P,\sigma,i\nu)$
extends to a unitary operator from $HS(H_{[P],\sigma})$ to $HS(H_{P,k\sigma})$.
\end{lemma}

\begin{proof}
For each $j=1,2,...$let $F_{j}$ be a finite subset of $\hat{K}$ such that
$F_{j}\subset F_{j+1}$, $F_{j}=F_{j}^{\ast}$ and $\cup_{j=1}^{\infty}F_{j}$ is
the set of elements of $\hat{K}$ such that $H_{[P],\sigma}[\gamma]\neq0 $. Fix
$F\subset\hat{K}$ and let $L=E_{F}TE_{F|_{H_{[P],\sigma}[F]}}$ and
$B=A_{k}(P,\sigma,i\nu)_{|_{H_{[P],\sigma}[F]}}$. Then%
\[
\mathrm{tr}(BLB^{-1}(BLB^{-1})^{\ast})=\mathrm{tr}(BLB^{-1}(B^{-1})^{\ast
}T^{\ast}B^{\ast})=\mathrm{tr}(B^{\ast}BLB^{-1}(B^{-1})^{\ast}L^{\ast}).
\]
By the definition of $\mu(\sigma,i\nu)$ we have%
\[
B^{\ast}B=\frac{C}{\mu(\sigma,i\nu)}I
\]
with $C$ the normalization necessary for the Harish-Chandra Plancherel
theorem. Thus
\[
\mathrm{tr}E_{F_{j}}C_{k}(P,\sigma,i\nu)(T)E_{F_{j}}\left(  E_{F_{j}}%
C_{k}(P,\sigma,i\nu)(T)E_{F_{j}}\right)  ^{\ast})=\mathrm{tr}(E_{F_{j}%
}TE_{F_{j}}T^{\ast}E_{F_{j}}).
\]
Observing that
\[
\mathrm{tr}(E_{F_{j}}TE_{F_{j}}T^{\ast}E_{F_{j}})=\sum_{\gamma_{l}\in F_{j}%
}\left\langle E_{F_{j}}T^{\ast}v_{i},E_{F_{j}}T^{\ast}v_{i}\right\rangle
\]
we can apply monotone convergence to conclude that
\[
\lim_{j\rightarrow\infty}\mathrm{tr}E_{F_{j}}C_{k}(P,\sigma,i\nu)(T)E_{F_{j}%
}\left(  E_{F_{j}}C_{k}(P,\sigma,i\nu)(T)E_{F_{j}}\right)  ^{\ast
})=\mathrm{tr}(TT^{\ast}).
\]

\end{proof}

Note that if $k_{1},k_{2}\in N_{K}(A_{P})$ then there exists a scalar
$c(k_{1},k_{2})$ such that $A_{k_{1}}(P,k_{2}\sigma,ik_{2}\nu)A_{k_{2}%
}(P,\sigma,\nu)=c(k_{1},k_{2})A_{k_{1}k_{2}}(P,\sigma,i\nu)$ we have

\begin{lemma}
\label{cocycle}$C_{k}(P,\sigma,i\nu)$ satisfies the cocycle condition%
\[
C_{k_{1}k_{2}}(P,x\sigma,i\nu)=C_{k_{1}}(P,k_{2}x\sigma,ik_{2}\nu)C_{k_{2}%
}(P,x\sigma,\nu)
\]
form$k_{1},k_{2},x\in N_{K}(A_{P})$.
\end{lemma}

Using the relationship between the intertwining operators and the
Harish-Chandra ${}^{o}C_{P|Q}$ operators (see (2) page 232 in \cite{RRGII}) we
see that Theorem 13.2.11 in \cite{RRGII} implies

\begin{theorem}
\label{maass-selberg}$(k,\nu)\mapsto C_{k}(P,\mu,i\nu)$ extends to a
continuous map from $N_{K}(A_{P})\times\mathfrak{a}_{P}^{\ast}$ to
$U(HS([P],\sigma))$ (the unitary operators on $HS([P],\sigma))$) that is real
analytic in $\nu$..
\end{theorem}

Set $N_{K}(A_{P})_{\sigma}=\{k\in N_{K}(A_{P})|k\sigma=\sigma\}$ (here we mean
$\sigma(k^{-1}xk)=\sigma(x)$ for all $x\in K$) and giving $N_{K}(A_{P})\sigma$
the manifold structure of $N_{K}(A_{P})/N_{K}(A_{P})_{\sigma}$ we have a
Hilbert vector bundle, $\mathbf{HS}([P],\sigma)$, over $N_{K}(A_{P}%
)\sigma\times\mathfrak{a}_{P}^{\ast}$ with fiber at $(k\sigma,\nu) $ given by
$HS(H_{\left[  P\right]  ,k\sigma})$. The above Lemma and Theorem imply that
we also have an $N_{K}(A_{P})$ action that is given as follows: If
$v\in\mathbf{HS}(H_{[P],\sigma})_{(k\sigma,i\nu)}$ and $x\in N_{K}(A)$ then
$xv=C_{x}(P,k\sigma,i\nu)v$. The fiberwise inner product on $\mathbf{HS}%
([P],\sigma)$ is given by the Hilbert-Schmidt inner product. This cocycle also
defines an $N_{K}(A_{P})$--Fr\'{e}chet bundle over $N_{K}(A_{P})\sigma
\times\mathfrak{a}_{P}^{\ast}$ with fiber $\mathcal{S}(HS([P],\sigma)$ which
we will denote \underline{$\mathcal{S}$}$(HS([P],\sigma)$.

\begin{lemma}
The $N_{K}(A_{P})$-invariant cross-sections, $\Gamma(\mathbf{HS}%
(H_{[P],\sigma}))^{N_{K}(A_{P})},,$ of $\mathbf{HS}(H_{[P],\sigma})$ are the
maps $\alpha:N_{K}(A_{P})\sigma\times\mathfrak{a}_{P}^{\ast}\rightarrow
\cup_{k\in N_{K}(A_{P})}HS([P],k\sigma)$ such that if $\mu\in N_{K}(A_{P})$
then $\alpha(\mu,\nu)\in HS([P],\mu),$ $\nu\mapsto\alpha(\mu,\nu)$ is
continuous and
\[
\alpha(k\mu,k\nu)=C_{k}(P,\mu,\nu)\alpha(\mu,\nu).
\]
$\Gamma(\underline{\mathcal{S}}(HS([P],\sigma))^{N_{K}(A_{P})}$ is the
subspace of $\alpha\in$ $\Gamma(\mathbf{HS}(H_{[P],\sigma}))^{^{N_{K}(A_{P})}%
}$ such that $\nu\mapsto\alpha(\sigma,\nu)$ is continuous with values in
\underline{$\mathcal{S}$}$(HS([P],\sigma)$.
\end{lemma}

\begin{proposition}
If $k\in N_{K}(A_{P}),\nu\in\left(  \mathfrak{a}_{P}^{\ast}\right)  ^{\prime}$
and $f\in\mathcal{C}(G)$ then
\[
C_{k}(P,\sigma,i\nu)(\pi_{P,\sigma,i\nu}(g))=\pi_{P,k\sigma,ik\nu}(g)
\]
and%
\[
C_{k}(P,\sigma,i\nu)(\pi_{P,\sigma,i\nu}(f))=\pi_{P,k\sigma,ik\nu}(f).
\]

\end{proposition}

\begin{proof}
The first equation is just a restatement of the intertwining property. The
second can be proved using the limiting argument in the end of the proof of
the preceding proposition.
\end{proof}

The next theorem is a compendium of results related to Harish-Chandra's
Plancherel Theorem: \cite{HCIII} Theorem 26.1 (c.f. \cite{RRGII} Theorem
12.7.1, Theorem 13.4.1)

\begin{theorem}
Let $F\subset\hat{K}$ be a finite set and let $\alpha:\mathfrak{a}_{P}^{\ast
}\rightarrow\mathrm{End}(H_{[P],\sigma}[F])$ be a Schwartz function then the
function $F_{\alpha}$ defined by%
\[
T(\alpha)(g)=\int_{\mathfrak{a}_{P}^{\ast}}\mathrm{tr}(\pi_{P,\sigma,i\nu
}(g^{-1})\alpha(\nu))\mu(\sigma,i\nu)d\nu
\]
is an element of $\mathcal{C}(G)$. Furthermore,
\[
\int_{G}\left\langle T(\alpha)(g),T(\alpha)(g)\right\rangle dg=\int
_{\mathfrak{a}_{P}^{\ast}}\mathrm{tr}(\alpha(\nu)\alpha(\nu)^{\ast})\mu
(\sigma,i\nu)d\nu.
\]

\end{theorem}

\begin{proof}
The first formula is \cite{RRGII} Theorem 12.7.1 and the second follows from
Theorem 13.4.1 using the fact that
\[
\varphi\ast\check{\varphi}(e)=\left\langle \varphi,\varphi\right\rangle .
\]
Here, as usual, $\check{\varphi}(x)=\overline{\varphi(x^{-1})}$ and $e$ is the
identity element of $G$.
\end{proof}

We are now ready to set up what we call the vector bundle version of the
abstract Plancherel Theorem for real reductive groups. Let $\mathcal{H}%
\mathbf{(}P,\sigma)$ be the space of all $\alpha\in\Gamma(\underline
{\mathcal{S}}(HS([P],\sigma))^{N_{K}(A_{P})}$ such that%
\[
\nu\mapsto\alpha(k\sigma,\nu)
\]
is Shwartz class, with values in $HS(H_{[P],k\sigma})$, and if $k\in
N_{K}(A_{P})$ then%
\[
\alpha(k\sigma,k\nu)=C_{k}(P,\sigma,\nu)(\alpha(\sigma,\nu))\text{.}%
\]
By the above, if $f\in\mathcal{C}(G)$ is $K$--finite then $((\mu,\nu
)\mapsto\pi_{P,\mu,i\nu}(f))\in\mathcal{H}\mathbf{(}P,\sigma)$.

Also for each $\nu\in\left(  \mathfrak{a}_{P}^{\ast}\right)  _{\mathbb{C}}$ we
look upon $HS(H_{[P],\sigma})$ as a representation of $G\times G$ with action%
\[
\Pi_{P,\sigma,\nu}(g,h)(T)=\pi_{P,\mu,\nu}(g)T\pi_{P,\mu,\nu}(h)^{-1},
\]
which is unitary if $\nu\in i\mathfrak{a}_{P}^{\ast}$. \cite{CWParam}
Proposition A.2 implies that $\mathcal{S}(H_{[P],\sigma})=HS(H_{[P],\sigma
})^{\infty}$ relative to the action of $K\times K$.

\begin{corollary}
Let $F_{j}$ be as in the proof of Lemma \ref{extension} and let $\alpha
\in\mathcal{H}\mathbf{(}P,\sigma)$. Then
\[
T(\alpha)=\lim_{j\rightarrow\infty}T(E_{F_{j}}\alpha E_{F_{j}})
\]
exists in $L^{2}(G)$ and
\[
\left\Vert T(\alpha)\right\Vert ^{2}=\int_{\mathfrak{a}_{P}^{\ast}}%
\mathrm{tr}(\alpha(\nu)\alpha(\nu)^{\ast})\mu(\sigma,i\nu)d\nu\text{.}%
\]

\end{corollary}

\begin{proof}
If $i>j$ then
\[
\left\Vert T(E_{F_{i}}\alpha E_{F_{i}})-T(E_{F_{j}}\alpha E_{F_{j}%
})\right\Vert =\left\Vert T(\left(  E_{F_{i}-F_{j}}\right)  \alpha
E_{F_{i}-F_{j}})\right\Vert .
\]
Thus%
\[
\left\Vert T(E_{F_{i}}\alpha E_{F_{i}})-T(E_{F_{j}}\alpha E_{F_{j}%
})\right\Vert ^{2}=\sum_{\gamma\in F_{i}-F_{j}}\left\Vert T(E_{\eta}\alpha
E_{\gamma})\right\Vert ^{2}.
\]
Note that%
\[
\left\Vert T(E_{\eta}\alpha E_{\gamma})\right\Vert ^{2}=\int_{\mathfrak{a}%
_{P}^{\ast}}\left\Vert E_{\gamma}\alpha(\nu)E_{\gamma}\right\Vert _{HS}^{2}%
\mu(\sigma,i\nu)d\nu.
\]
Given $\varepsilon>0$ there exists $n$ such that if $\gamma\in\hat{K}-F_{n}$
then $\left\Vert E_{\gamma}\alpha(\nu)E_{\gamma}\right\Vert _{HS}^{2}%
\leq\varepsilon\left\Vert \alpha(v)\right\Vert _{HS}^{2}.$ Thus, if $j\geq n$
then%
\[
\left\Vert T(E_{F_{i}}\alpha E_{F_{i}})-T(E_{F_{j}}\alpha E_{F_{j}%
})\right\Vert ^{2}\leq\varepsilon\int_{\mathfrak{a}_{P}^{\ast}}\left\Vert
\alpha(\nu)\right\Vert _{HS}^{2}\mu(\sigma,i\nu)d\nu.
\]
Now, by definition of the Schwartz space, we have for each $d$
\[
\left\Vert \alpha(\nu)\right\Vert _{HS}^{2}\leq C_{d}(1+\left\Vert
\nu\right\Vert )^{-d}%
\]
and since
\[
\mu(\sigma,i\nu)\leq L(1+\left\Vert \nu\right\Vert )^{r}%
\]
for some $r.$ If we take $d$ sufficiently large and if $j\geq n$ then%
\[
\left\Vert T(E_{F_{i}}\alpha E_{F_{i}})-T(E_{F_{j}}\alpha E_{F_{j}%
})\right\Vert ^{2}\leq\varepsilon C_{d}L\int_{\mathfrak{a}_{P}^{\ast}%
}(1+\left\Vert \nu\right\Vert )^{r-d}.
\]
So the sequence $\left\{  T(E_{F_{i}}\alpha E_{F_{i}})\right\}  $. is Cauchy.
The above argument also proves the second formula.
\end{proof}

On $N_{K}(A_{P})\sigma\times\mathfrak{a}_{P}^{\ast}$ we use the product of the
invariant normalized measure on $N_{K}(A_{P})/N_{K}(A_{p})_{\sigma}$ and
Lebesgue measure on $\mathfrak{a}_{P}^{\ast}$. If $f\in\mathcal{C}(G),\nu
\in\mathfrak{a}_{P}^{\ast}$ set
\[
S(f)(\mu,\nu)=S_{P,\sigma}(f)(\mu,\nu)=\pi_{P\mu,i\nu}(f)
\]
for $\mu\in N_{K}(A_{P})\sigma$.

\begin{theorem}
\label{density}Denote by $L^{2}(G)_{[P],[\sigma]}$ the closure of
$T(\mathcal{H}\mathbf{(}P,\sigma))$ in $L^{2}(G).$ Let $\mathcal{C}%
(G)_{[P],\sigma}$ be the closure of the span of the functions $T(\alpha)$ for
$\alpha$ a Schartz function with values in $\cup_{j}End(H_{[P],\sigma}%
[F_{j}])$ in $\mathcal{C}(G)$. Then $S(\mathcal{C}(G)_{[P],\sigma})$ is dense
in the space of $N_{K}(A_{P})$ invariant $L^{2}$ cross-sections of
$\mathbf{HS}([P],\sigma)$.
\end{theorem}

\begin{proof}
Theorem 13.3.2 in \cite{RRGII} with the erroneous $\Psi$ replaced with $\Phi$
(which is a restatement of Theorem 26.1 in \cite{HCIII})implies that if
$\alpha\in\mathcal{H}\mathbf{(}P,\sigma)$ then $S(T(\alpha))(\mu,\nu
)=\alpha(\mu,\nu)$. This implies the density.
\end{proof}

This theorem can be interpreted as an explicit version of the abstract
Plancherel Theorem for Real Reductive Groups (see section 14.12 in
\cite{RRGII}). Indeed, set $\Gamma^{2}(\mathbf{HS}([P],\sigma))$ equal to the
Hilbert space of elements completion of the space of $\alpha\in\Gamma
(\mathbf{HS}([P],\sigma))$ such that
\[
\int_{\mathfrak{a}_{P}^{\ast}}\mathrm{tr}(\alpha(\sigma,\nu)\alpha(\sigma
,\nu)^{\ast})\mu(\sigma,i\nu)d\nu<\infty
\]
relative to the inner product
\[
\left\langle \alpha,\beta\right\rangle =d(\sigma)\int_{\mathfrak{a}_{P}^{\ast
}}\mathrm{tr}(\alpha(\sigma,\nu)\beta(\sigma,\nu)^{\ast})\mu(\sigma,i\nu)d\nu.
\]

\begin{theorem}
\label{Abs-Planch}Let $\mathcal{P}(G)$ be the set of associativity classes of
cuspidal, parabolic subgroups. For each $[P]\in\mathcal{P}(G)$ (denoted by a
representative that is standard) let $M_{[P]}$ be $M_{P}$, the standard Levi
factor. If $[P]\in\mathcal{P}(G)$ and $\sigma,\mu$ are irreducible square
integrable representations of $^{o}M_{[P]}$ then we say that $\sigma\sim\mu$
if $\sigma$ is unitarily equivalent with $k\mu$ for some $k$ in $N_{K}%
(A_{P}).$ Let $\overline{\mathcal{E}}_{2}(^{o}M_{[P]})$ be the set of such
equivalence classes of square integrable, irreducible representations. If
$\left[  \sigma\right]  \in\overline{\mathcal{E}}_{2}(^{o}M_{[P]})$ with
$\sigma$ a fixed representative. Then as a representation of $G\times G$,
$L^{2}(G)$ is equivalent to the Hilbert direct sum
\[
\bigoplus_{\lbrack P]\in\mathcal{P}(G)}\bigoplus_{[\sigma]\in\overline
{\mathcal{E}}_{2}(^{o}M_{[P]})}\Gamma^{2}(\mathbf{HS}([P],\sigma
))^{N_{K}(A_{P})}.
\]
The unitary intertwining operator giving the equivalence is given on
$\mathcal{C}(G)$ by%
\[
S(f)=\sum_{[P]\in\mathcal{P}(G)}\sum_{[\sigma]\in\overline{\mathcal{E}}%
_{2}(^{o}M_{[P]})}S_{P,\sigma}(f).
\]

\end{theorem}

We next describe the $C^{\infty}$ vectors for the representation of $G\times
G$ on $\Gamma^{2}(\mathbf{HS}([P],\sigma))$ and their intersection with
$\mathcal{C}(G)$. For this we need

\begin{lemma}
Let $C$ be the Casimir operator on $G$ corresponding to $B$ on $\mathfrak{g}$
and let $C_{K}$ be the Casimir operator on $K$ corresponding to
$B_{|\mathfrak{k}}$. If $f\in C^{\infty}(G)$ set%
\[
q_{p,r,s}(f)=\left\Vert C^{p}L_{C_{K}^{r}}R_{C_{K}^{sp,}}f\right\Vert .
\]
If $p,q,r\in\mathbb{Z}_{\geq0}$ and if $L^{2}(G)^{\infty}$ denotes the
$G\times G-C^{\infty}$ vectors of $L^{2}(G)$ then
\[
L^{2}(G)^{\infty}=\left\{  f\in C^{\infty}(G)|q_{p,r,s}(f)<\infty
\mathrm{\ all\ }p,q,r\in\mathbb{Z}_{\geq0}\right\}
\]
endowed with the topology induced by the seminorms $q_{p,r,s}$.
\end{lemma}

\begin{proof}
Let $\Delta=C-2C_{K}$ then $\Delta$ defines an elliptic operator on $G$ using
elliptic regularity one can show (see \cite{CWParam} Lemma A.1) that
$L^{2}(G)^{\infty}$ defined as above using the seminorms $\xi_{p,q}%
(f)=\left\Vert L_{\Delta^{p}}R_{\Delta^{q}}f\right\Vert $, Using the formula
for $\Delta$ we see that if $q_{p,q,r}(f)<\infty$ all $p,q,r$ then $\xi
_{p,q}(f)<\infty$ all $p,q.$ Thus if $Z$ is the Fr\'{e}chet space defined by
the $q$--seminorms then $L^{2}(G)^{\infty}\subset Z$. On the other hand the
usual definition of $C^{\infty}$ vector implies that $Z\subset L^{2}%
(G)^{\infty}$. The closed graph theorem now implies that $Z=L^{2}(G)^{\infty}$
as a Fr\'{e}chet space.
\end{proof}

If $\alpha\in\Gamma(\mathcal{\underline{\mathcal{S}}}(HS([P],\sigma)))$ then
we set
\[
w_{p,q,r}(\alpha)=\mathrm{sup}\left\{  ((1+\left\Vert \nu\right\Vert
)^{r}\left\Vert C_{K}^{p}\alpha(k\sigma,\nu)C_{K}^{r}\right\Vert |k\in
N_{K}(A_{P}),\nu\in\mathfrak{a}_{P}^{\ast}\right\}  .
\]

Set $\Gamma_{rd}(\mathcal{\underline{\mathcal{S}}}(HS([P],\sigma)))$ equal to
the space of all $\alpha\in\Gamma(\mathcal{\underline{\mathcal{S}}%
}(HS([P],\sigma)))$ such that $w_{p,q,r}(\alpha)<\infty$ for all
$p,q,r\in\mathbb{Z}_{\geq0}$ endowed with the topology induced by the
seminorms $w_{p.q.r\text{.}}$

\begin{theorem}
\label{smoothdecomp}$T$ maps $\Gamma_{rd}(\mathcal{\underline{\mathcal{S}}%
}(HS([P],\sigma)))$ to $L^{2}(G)_{[P],\sigma}^{\infty}$ defines an isomorphism
of $G\times G$--Fr\'{e}chet representations. Also $L^{2}(G)^{\infty}$ is the
isomorphic with
\[
\bigoplus_{\lbrack P]\in\mathcal{P}(G)}\bigoplus_{[\sigma]\in\overline
{\mathcal{E}}_{2}(^{o}M_{[P]})}\Gamma_{hs}^{2}(\underline{\mathcal{S}%
}\mathbf{HS}([P],\sigma))^{N_{K}(A_{P})}%
\]
Fr\'{e}chet space direct sum as a smooth Fr\'{e}chet representation of
$G\times G$.
\end{theorem}

\begin{proof}
The second assertion follows from the first and the fact that $L^{2}(G)$ is
isomorphic as a representation of $G\times G$ with $\bigoplus_{[P],[\sigma
]}L^{2}(G)_{[P],\sigma}$. We will now prove the first assertion. We note that
if $\mu$ is an irreducible representation of ${}^{o}M_{P}$ then if $\nu
\in\mathfrak{a}_{P}^{\ast}$%
\[
\pi_{P,\mu,i\nu}(C)=\lambda_{\mu}-\left\langle \nu,\nu\right\rangle
-\left\langle \rho,\rho\right\rangle
\]
with $\pi_{\mu}(C)=\lambda_{\mu}I$ .\ Also, $\lambda_{k\sigma}=\lambda
_{\sigma},$ $k\in N_{K}(A_{P})$. This implies that if $f\in\mathcal{C}(G)$
then
\[
\pi_{P,k\sigma,i\nu}((\lambda_{\sigma}-\left\langle \rho_{P},\rho
_{P}\right\rangle -C)^{r}L_{C_{K}^{p}}R_{C_{K}^{q}}f)=\left\Vert
\nu\right\Vert ^{2r}\pi_{\lbrack P],k\sigma}(C_{K}^{p})\pi_{P,\,k\sigma,i\nu
}(f)\pi_{\lbrack P],k\sigma}(C_{K}^{r}).
\]
This implies that the map $S_{[P],\sigma|L^{2}(G)_{[P],\sigma}\cap
\mathcal{C}(G)}$ extends to a continuous map of $L^{2}(G)_{[P],\sigma}%
^{\infty}$ into $\Gamma_{rd}(\mathcal{\underline{\mathcal{S}}}(HS([P],\sigma
)))^{N_{K}(A_{P})}$. The equality also implies that the map $T $ from the
dense space $S_{[P],\sigma}(\mathcal{C}(G))$ in $\Gamma^{2}(\mathbf{HS}%
([P],\sigma)$ extends to a continuous map of $\Gamma_{rd}(\mathcal{\underline
{\mathcal{S}}}(HS([P],\sigma)))$ to $L^{2}(G)_{[P],\sigma}^{\infty}$. Since
$ST$ and $TS$ are identity maps on the respective $L^{2}$ spaces the result follows.
\end{proof}

\section{The Jacquet integrals and the Whittaker Schwartz space}

Notation as in the preceding section. We will also write $H_{\bar{P}%
,\sigma,\nu}$ for $H_{[P],\sigma}$ if we are using the space as the
representation space for $\pi_{\bar{P},\sigma,\nu}$. In this section we will
recall some properties of the Jacquet integrals. Let $\chi$ be a unitary
character of $N_{o}$. Note that $\ker d\chi\supset\lbrack\mathfrak{n}%
_{o},\mathfrak{n}_{o}]$. We say that $\chi$ is generic if $d\chi$ restricted
to any weight space of $A_{o}$ on $\mathfrak{n}_{o}/[\mathfrak{n}%
_{o},\mathfrak{n}_{o}]$ is non-zero. If $(\pi,V)$ is a smooth Fr\'{e}chet
representation of $G$ then define $Wh_{\chi}(V)$ to be the subspace of
$\lambda\in V^{\prime}$ such that $\lambda(\pi(n)v)=\chi(n)\lambda(v)$, $n\in
N_{o}$, $v\in V$. Fix $\chi$, a generic character of $N_{o}.$

Let $(\sigma,H_{\sigma})$ be an irreducible square integrable representation
of a real reductive group $M$ , $N$ a maximal unipotent subgroup of $M$ ,
$\eta$ \ a generic character of $N$ and assume that $Wh_{\eta}(H_{\sigma
}^{\infty})\neq0$. In \cite{WPT} Corollary 38 we defined an inner product on
$Wh_{\eta}(H_{\sigma}^{\infty})$, $(...,...)_{\sigma}$ with the following
property (here we have fixed Haar measures on $M$ and $N$ and have the
quotient measure $d\bar{m}$ on $N\backslash M$)%
\[
\int_{N\backslash M}\lambda(\sigma(m)v)\overline{\mu(\sigma(m)w)}d\bar
{m}=(\lambda,\mu)_{\sigma}\left\langle v,w\right\rangle
\]
for $\lambda,\mu\in Wh_{\eta}(H_{\sigma}^{\infty})$ and $v,w\in H_{\sigma
}^{\infty}$.

Let $P$ be a standard cuspidal parabolic subgroup of $G$ and let
$(\sigma,H_{\sigma})$ be an irreducible square integrable representation of
${}^{o}M_{P}$. If $\lambda\in Wh_{\chi_{|N_{o}\cap M_{P}}}(H_{\sigma}^{\infty
})$ and $\nu\in\left(  \mathfrak{a}_{P}^{\ast}\right)  _{\mathbb{C}}$, $u\in
H_{[P],\sigma}^{\infty}$ then consider the integral%
\[
J_{\chi}(P,\sigma,\nu)(\lambda)(u)=\int_{N_{P}}\chi(n)^{-1}\lambda({}_{\bar
{P}}u_{\sigma,\nu}(n))dn.
\]
This integral converges for $\nu$ in the open half space
\[
\left(  \mathfrak{a}_{P}^{\ast}\right)  _{\mathbb{C}}^{-}=\{\nu\in\left(
\mathfrak{a}_{P}^{\ast}\right)  _{\mathbb{C}}|\operatorname{Re}(\nu
,\alpha)<0,\alpha\in\Phi(P_{o},A_{o})\}
\]
of $\left(  \mathfrak{a}_{P}^{\ast}\right)  _{\mathbb{C}}$ and defines a
weakly holomorphic map of $\left(  \mathfrak{a}_{P}^{\ast}\right)
_{\mathbb{C}}^{-}$ into $\left(  H_{[p],\sigma}^{\infty}\right)  ^{\prime}$.
If $\nu\in\left(  \mathfrak{a}_{P}^{\ast}\right)  _{\mathbb{C}}^{-}$ then
$J_{\chi}(P,\sigma,\nu)(\lambda)\in Wh_{\chi}(H_{\bar{P},\sigma,\nu}^{\infty})
$ (see Lemma 15.6.5 \cite{RRGII})Theorem 15.6.7 in \cite{RRGII}) asserts

\begin{theorem}
\label{Isomorphism}If $\lambda\in Wh_{\chi_{|N_{o}\cap M_{P}}}(H_{\sigma
}^{\infty})$ then $\nu\mapsto J_{\chi}(P,\sigma,\nu)(\lambda)$ has a weakly
holomorphic extension to $\left(  \mathfrak{a}_{P}^{\ast}\right)
_{\mathbb{C}} $. Furthermore, for all $\nu\in\left(  \mathfrak{a}_{P}^{\ast
}\right)  _{\mathbb{C}}$
\[
J_{\chi}(P,\sigma,\nu):Wh_{\chi_{|N_{o}\cap M_{P}}}(H_{\sigma}^{\infty
})\rightarrow Wh_{\chi}(H_{\bar{P},\sigma,\nu}^{\infty})
\]
is bijective.
\end{theorem}

We also recall Theorems 43 and 45 in \cite{WPT}.

\begin{theorem}
\label{Tempered-Estimate}There exists $m>0$ and a continuous seminorm,
$\gamma_{1}$ on $I_{\sigma}^{\infty}$ such that if $u\in I_{\sigma}^{\infty}$
and $\lambda\in Wh_{\chi_{|N_{o}\cap M_{P}}}(H_{\sigma}^{\infty})$ then
\[
\left\vert J(P,\sigma,i\nu)(\lambda)(u)\right\vert \leq\gamma_{1}%
(u)(1+\left\Vert \nu\right\Vert )^{m}\left\Vert \lambda\right\Vert _{\sigma}%
\]
for $\nu\in\mathfrak{a}^{\ast}$.
\end{theorem}

\begin{theorem}
\label{J-estimate}Let $\omega\subset\mathfrak{a}_{P}^{\ast}$ be a compact set
and let $u\in(I_{\sigma}^{\infty})_{K}$then there exist $r>0$ , C$_{u} $, and
$d$ such that
\[
\left\vert J(P,\sigma,i\nu)(\lambda)(\pi_{\bar{P},\sigma,i\nu}%
(ak)u)\right\vert \leq C_{u}\left\Vert \lambda\right\Vert _{\sigma}a^{\rho
_{o}}(1+\left\Vert \log a\right\Vert )^{d}%
\]
for $\nu\in\mathfrak{\omega}$ and $a\in A_{o},k\in K$.
\end{theorem}

The Jacquet integrals that came into the Whittaker Plancherel Formula in
\cite{WPT} were $J_{\chi^{-1}}$ not $J_{\chi}$ the following (trivial) lemma
(which will be used in the direct integral form of the\ Whittaker Plancherel
Theorem) is part of an explanation of this (Recall $(L_{n}\varphi
)(g)=\varphi(n^{-1}g)$ and $\pi_{\bar{P},\sigma,i\nu}(L_{n}\varphi)=\pi
_{\bar{P},\sigma,i\nu}(n)\pi_{\bar{P},\sigma,i\nu}(\varphi)$.)

\begin{lemma}
\label{observation}$\varphi\in\mathcal{C}(G)$, $\lambda\in Wh_{\chi
_{|N_{o}\cap M_{P}}^{-1}}(H_{\sigma}^{\infty}),u\in H_{[P],\sigma}^{\infty}$
$,\nu\in\mathfrak{a}^{\ast}$ and $n\in N_{o}$ then%
\[
J_{\chi^{-1}}(P,\sigma.\iota\nu(\lambda)(\pi_{\bar{P},\sigma.i\nu}((L_{n^{-1}%
}-\chi(n))\varphi)u)u)=0.
\]

\end{lemma}

\begin{proof}%
\[
(\pi_{\bar{P},\sigma.i\nu}(L_{n^{-1}}\varphi)u)=\pi_{\bar{P},\sigma.i\nu
}(n)^{-1}\pi_{\bar{P},\sigma.i\nu}(\varphi)u
\]
so%
\[
J_{\chi^{-1}}(P,\sigma.\iota\nu(\lambda)(\pi_{\bar{P},\sigma.i\nu}(L_{n^{-1}%
}\varphi)u)=\chi(n)J_{\chi^{-1}}(P,\sigma.\iota\nu(\lambda)(\pi_{\bar
{P},\sigma.i\nu}(\varphi)u.
\]

\end{proof}

Fix once and for all a Haar measure on $N_{o}$ and thus a right invariant
measure on $N_{o}\backslash G$. \ Recall that the Whittaker Schwartz space,
$\mathcal{C}(N_{o}\backslash G,\chi)$, for $\chi$ a regular unitary character
of $N_{o}$ , is the set of elements $f\in C^{\infty}(G)$ such that
\[
q_{d,x}(f)=\sup_{g\in G}(a(g)^{-\rho_{o}}(1+\log\left\Vert a(g)\right\Vert
)^{d}|xf(g)|<\infty
\]
for all $x\in U(\mathfrak{g})$ and $d\in\mathbb{Z}$ endowed with the topology
defined by the seminorms $q_{d,x}$.

The following is Proposition 14 in \cite{WPT}

\begin{proposition}
\label{Fourier}If $f\in\mathcal{C}(G)$ then the the integral
\[
f_{\chi}(g)=\int_{N_{o}}\chi(n)^{-1}f(ng)dn
\]
converges absolutely and uniformly on compacta in $g$ to an element of
$\mathcal{C(}N_{o}\backslash G;\chi)$. Furthermore, the map defined by
$T_{\chi}(f)=f_{\chi}$ is a continuous map from $\mathcal{C}(G)$ to
$\mathcal{C(}N_{o}\backslash G;\chi)$.
\end{proposition}

\begin{lemma}
\label{Basic}If $\varphi\in\mathcal{C}(G),\sigma$ is a square integrable,
irreducible representation of ${}^{o}M_{P}$, $\lambda\in Wh_{\chi_{|N_{o}\cap
M_{P}}^{-1}}(H_{\sigma}^{\infty})$ and $\nu\in\mathfrak{a}_{P}^{\ast}$, $u\in
H_{[P],\sigma}^{\infty}$
\[
J_{\chi^{-1}}(P,\sigma,\iota\nu)(\lambda)(\pi_{\bar{P},\sigma,i\nu}%
(\varphi)u)=\int_{N_{o}\backslash G}J_{\chi^{-1}}(P,\sigma,\iota\nu
)(\lambda)(\pi_{\bar{P},\sigma,i\nu}(g)u)\varphi_{\chi}(g)d\bar{g}.
\]

\end{lemma}

\begin{proof}
Using Theorem \ref{J-estimate} and Proposition \ref{Fourier} we see that both
sides of the equality that we are proving are continuous in $\varphi$. It is
therefore enough to prove the result for $\varphi\in C_{c}^{\infty}(G)$. For
such $\varphi$, the function%
\[
\nu\mapsto J_{\chi^{-1}}(P,\sigma,\nu)(\lambda)(\pi_{\bar{P},\sigma,\nu
}(\varphi)u)
\]
and
\[
\nu\mapsto\int_{N_{o}\backslash G}J_{\chi^{-1}}(P,\sigma,\nu)(\lambda
)(\pi_{\bar{P},\sigma,\nu}(g)u)\varphi_{\chi}(g)d\bar{g}%
\]
are holomorphic in $\nu$ on $\left(  \mathfrak{a}_{P}^{\ast}\right)
_{\mathbb{C}}$ \ Thus, it is enough to prove the result for $\nu\in\left(
\mathfrak{a}_{P}^{\ast}\right)  _{\mathbb{C}}^{-}$. We also note that there
exists $z\in\left(  H_{\sigma}\right)  _{K\cap M_{P}}$ such that
\[
\lambda(w)=\int_{N_{o}\cap M_{P}}\chi(n_{1})\left\langle \sigma(n_{1}%
)w,z\right\rangle dn_{1}%
\]
(\cite{WPT} Theorem 37). For simplicity we will write
\[
{}_{\bar{P}}w_{\sigma,\nu}=w_{\nu}%
\]
We calculate, using standard integration formulas (c.f. Lemma 2.4.5
\cite{RRGI}),for $w\in H_{[P],\sigma}^{\infty}$
\[
J_{\chi^{-1}}(P,\sigma,\nu)(\lambda)(\pi_{\bar{P},\sigma,\nu}(\varphi
)u)=\int_{N_{P}}\chi(n)\lambda(\left(  \pi_{\bar{P},\sigma,\nu}(\varphi
)u\right)  _{\nu}(n))dn
\]%
\[
=\int_{N_{P}\times N_{o}\cap M_{P}}\chi(n)\chi(n_{1})\left\langle \sigma
(n_{1})\left(  \pi_{\bar{P},\sigma,\nu}(\varphi)u\right)  _{\nu}%
(n)),z\right\rangle dn
\]%
\[
=\int_{N_{P}\times N_{o}\cap M_{P}}\chi(n)\chi(n_{1})\left\langle \left(
\pi_{\bar{P},\sigma,\nu}(\varphi)u\right)  _{\nu}(n_{1}n)),z\right\rangle
dn_{1}dn
\]%
\[
=\int_{N_{o}}\chi(n_{o})\left\langle \left(  \pi_{\bar{P},\sigma,\nu}%
(\varphi)u\right)  _{\nu}(n_{o})),z\right\rangle dn_{o}%
\]
(here we have chosen $dn_{o}$ so that it is the push forward of $dn_{1}dn$ on
$N_{o}\cap M_{P}\times N_{P}$)%
\[
=\int_{N_{o}}\chi(n_{o})\left\langle \int_{N_{o}\times A_{o}\times K}%
a^{-2\rho_{o}}\varphi(nak)u_{\nu}(n_{o}nak)dndadk,z\right\rangle dn_{o}%
\]
since $\varphi$ has compact support we can integrate this as a quadruple
integral in any order. We integrate over $n_{o}$ first an get%
\[
\int_{N_{o}\times N_{o}\times A_{o}\times K}\chi(n_{o}n^{-1})\varphi
(nak)a^{-2\rho_{o}}\left\langle u_{\nu}(n_{o}ak),z\right\rangle dndadkdn_{o}.
\]
We next integrate over $n$ and get%
\[
\int_{N_{o}\times A_{o}\times K}\chi(n_{o})\varphi_{\chi}(ak)a^{-2\rho_{o}%
}\left\langle u_{\nu}(n_{o}ak),z\right\rangle dn_{o}dadk
\]%
\[
\int_{A_{o}\times K}a^{-2\rho_{o}}\varphi_{\chi}(ak)\int_{N_{o}}\chi
(n_{o})\left\langle u_{\nu}(n_{o}ak),z\right\rangle dn_{o}dadk
\]%
\[
=\int_{A_{o}\times K}a^{-2\rho_{o}}\varphi_{\chi}(ak)J_{\chi^{-1}}(\bar
{P},\sigma,\nu)(\lambda)(\pi_{\bar{P},\sigma,\nu}(ak)u)ak)
\]%
\[
=\int_{N_{o}\backslash G}J_{\chi^{-1}}(P,\sigma,\nu)(\lambda)(\pi_{\bar
{P},\sigma,\nu}(g)u)\varphi_{\chi}(g)d\bar{g}%
\]
here we used the fact that up to normalization of measures
\[
\int_{N_{o}\backslash G}f(g)d\bar{g}=\int_{A_{o}\times K}a^{-2\rho_{o}%
}f(ak)dadk
\]
if $f(ng)=f(g),n\in N_{o},g\in G$.
\end{proof}

\section{A technical lemma and a key result\label{key-result}}

Before we set the Whittaker analog of the material in Section
\ref{G-Plancherel Space} we need some preliminary results. Let $H$ be a
Hilbert space with inner product $\left\langle ...,...\right\rangle $ and let
$H^{\prime}$ be the, continuous, dual space with the dual inner product. If
$v\in H$ then let $v^{\#}\in H^{\prime}$ be defined by $v^{\#}(x)=\left\langle
x,v\right\rangle .$ Then the map $v\mapsto v^{\#}$ is a conjugate linear
isometry of $H$ onto $H^{\prime}$. If $H_{1}$ and $H_{2}$ are Hilbert spaces
and $T\ $\ is a continuous linear map from \thinspace$H_{1}$ to $H_{2} $then
define $T^{t}$ as usual (if $\lambda\in H_{2}^{\prime}$, $T^{t}(\lambda
)=\lambda\circ T\in H_{1}^{\prime}$). The following is obvious

\begin{lemma}
If $v\in H_{2}$ then $T^{t}\bar{v}=\left(  T^{\ast}v\right)  ^{\#}$ here, as
usual, $T^{\ast}is$ the Hilbert Space adjoint of $T$.
\end{lemma}

Also identifying $(H^{\prime})^{\prime}$ with $H$ we have $v^{\#\#}=v.$

Keep the notation of Section \ref{G-Plancherel Space}. Let $\sigma
=\sigma_{\omega},\omega\in\mathcal{E}_{2}({}^{o}M_{P})$ and $\lambda
_{1},...,\lambda_{d}$ be an orthonormal basis of $Wh_{\chi_{|N_{o}\cap M_{P}}%
}(H_{\sigma}^{\infty})$ with respect to the inner product $(...,...)_{\sigma}$
in Corollary 38 of \cite{WPT}. Define for $\alpha\in C_{c}^{\infty
}(\mathfrak{a}^{\ast})$, $u\in I_{\sigma}$ and $\nu\in\mathfrak{a}_{P}^{\ast
}.$%
\[
\Lambda_{\chi}(\sigma,\alpha,u)(\nu)=\alpha(\nu)\sum_{i=1}^{d}J_{\chi
}(P,\sigma,i\nu)(\lambda_{i})(u)\lambda_{i}^{\#},\nu\in\mathfrak{a}^{\ast}%
\]
and%
\[
\Lambda_{\chi,\sigma}:C_{c}^{\infty}(\mathfrak{a}^{\ast})\otimes I_{\sigma
}\rightarrow C_{c}^{\infty}(\mathfrak{a}^{\ast},Wh_{\chi_{|N_{o}\cap M_{P}}%
}(H_{\sigma}^{\infty})^{\ast})
\]
by%
\[
\Lambda_{\chi,\sigma}(\alpha\otimes u)=\Lambda_{\chi}(\sigma,\alpha,u).
\]

The following result is, essentially, Lemma 15.8.6 in \cite{RRGII}.

\begin{lemma}
\label{surjective}The map $\Lambda_{\chi,\sigma}$ is surjective.
\end{lemma}

\begin{proof}
It is enough to prove that if $\beta\in C_{c}^{\infty}(\mathfrak{a}^{\ast})$
then $\beta\lambda^{\#}$ is in the image of $\Lambda_{\sigma}$.

If $\nu\in\mathfrak{a}^{\ast}$ then the map
\[
J_{\chi}(P,\sigma,i\nu):Wh_{\chi_{|N_{o}\cap M_{P}}}(H_{\sigma}^{\infty
})\rightarrow Wh_{\chi}(I_{P,\sigma,i\nu}^{\infty})
\]
is bijective. Thus there exist $w_{1}^{\nu},...,w_{d}^{\nu}\in I_{\sigma}$
such that
\[
J_{\chi}(P,\sigma,i\nu)(\lambda_{i})(w_{j}^{\nu})=\delta_{ij}\text{.}%
\]
Continuity implies that there exists an open neighborhood of $\nu,$ $U_{\nu
}\subset\mathfrak{a}^{\ast}$, such that%
\[
\det\left[  J_{\chi}(P,\sigma,i\mu)(\lambda_{i})(w_{j}^{\nu})\right]
\neq0,\mu\in U_{\nu}\text{.}%
\]
The covering $\{U_{\nu}\}_{\nu\in\mathrm{supp}(\beta)}$ of $\mathrm{supp}%
(\beta)$ has a finite subcover $\{U_{\nu_{i}}\}_{i=1}^{r}$. Set $U_{i}%
=U_{\nu_{i}}$ and $w_{j}^{i}=w_{j}^{\nu_{i}}$. Let $\varphi_{1},...,\varphi
_{m}$ be a partition of unity subordinate to the covering $\{U_{i}\}_{i=1}%
^{r}$ in particular for each $j$ there exists $i_{j}$ such that $\mathrm{supp}%
\varphi_{j}\subset U_{i_{j}}$. If $\mu\in U_{i_{l}}$ set
\[
\varphi_{l}(\mu)\left[  J_{\chi}(P,\sigma,i\mu)(\lambda_{i})(w_{j}^{i_{l}%
})\right]  ^{-1}=[z_{ij}^{l}(\mu)]
\]
and extend $z_{ij}^{l}$ to $\mathfrak{a}^{\ast}$ be zero. Then $z_{ij}^{l}\in
C_{c}^{\infty}(\mathfrak{a}^{\ast})$. Note that
\[
\sum_{j=1}^{d}z_{ji}^{l}(\mu)J_{\chi}(P,\sigma,i\mu)(\lambda_{k})(w_{j}%
^{l})=\delta_{ik}\varphi_{l}.
\]
Thus
\[
\sum_{l=1}^{m}\sum_{j=1}^{d}\Lambda_{\chi}(\sigma,w_{j}^{l},z_{ji}^{l}%
\beta)(\mu)=\sum_{l=1}^{m}\sum_{j=1}^{d}\sum_{k=1}^{d}z_{ji}^{l}(\mu
)\gamma(\mu)J_{\chi}(P,\sigma,i\mu)(\lambda_{k})(w_{j}^{i_{l}})\lambda
_{k}^{\#}%
\]%
\[
=\beta(\mu)\sum_{l=1}^{r}\sum_{k=1}^{d}\varphi_{l}(\mu)\delta_{ik}\lambda
^{\#}=\beta(\mu)\lambda^{\#}.
\]

\end{proof}

\begin{theorem}
Notation as in the previous lemma. If $\alpha\in C_{c}^{\infty}(\mathfrak{a}%
^{\ast}),\mu\in Wh_{\chi_{|N_{o}\cap M_{P}}}(H_{\sigma}^{\infty}),u\in\left(
H_{[P],\sigma}\right)  _{K}$ and if $f$ is defined by%
\[
f(g)=\int_{\mathfrak{a}^{\ast}}\alpha(\nu)J_{\chi}(P,\sigma,i\nu)(\mu
)(\pi_{\bar{P},\sigma,i\nu}(g)u)\mu(\sigma,i\nu)d\nu
\]
then $f\in\mathcal{C}(N_{o}\backslash G,\chi)$.
\end{theorem}

\begin{proof}
The above lemma implies that $\alpha\mu$ is a linear combination of functions
of the form $\overline{\Lambda_{\sigma}(\bar{\beta}\otimes w)}$ with $\beta\in
C_{c}^{\infty}(\mathfrak{a}^{\ast})$ and $w\in I_{\sigma}$. So it is enough to
show that a function of the form
\[
h(g)=\int_{\mathfrak{a}^{\ast}}J_{\chi}(P,\sigma,i\nu)(\overline
{\Lambda_{\sigma}(\bar{\beta}\otimes w)}))(\pi_{\bar{P},\sigma,i\nu}%
(g)u)\mu(\sigma,i\nu)d\nu
\]
is in $\mathcal{C}(N_{o}\backslash G,\chi)$. If we write out the definition of
$\overline{\Lambda_{\sigma}(\bar{\beta}\otimes w)}$ we have%
\[
h(g)=\int_{\mathfrak{a}^{\ast}}\sum_{i=1}^{d}\beta(\nu)\overline{J_{\chi
}(P,\sigma,i\nu)(\lambda_{i})(w)}J_{\chi}(P,\sigma,i\nu)(\lambda_{i}%
)((\pi_{\bar{P},\sigma,i\nu}(g)u)\mu(\sigma,i\nu)d\nu.
\]
Define $\psi$ by%
\[
\psi(g)=\int_{\mathfrak{a}^{\ast}}\beta(\nu)\left\langle \pi_{\bar{P}%
,\sigma,i\nu}(g)u,w\right\rangle \mu(\sigma,i\nu)d\nu.
\]
Then Theorem 12.7.1 \cite{RRGII} implies that $\psi\in\mathcal{C}(G)$. So
$\psi_{\chi}\in\mathcal{C}(N_{o}\backslash G,\chi)$. Theorem 57 \cite{WPT}
implies that $h=\psi_{\chi}$. Completing the proof.
\end{proof}

\section{The Whittaker Plancherel space}

We look upon $HS(H_{[P],\sigma})$ as the Hilbert space tensor product
\[
H_{[P],\sigma}\otimes_{HS}H_{[P].\sigma}^{\prime}.
\]
The $K\times K-C^{\infty}$ vectors are the Fr\'{e}chet space tensor product%
\[
H_{[P],\sigma}^{\infty}\otimes_{F}\left(  H_{[P],\sigma}^{^{\prime}}\right)
^{\infty}%
\]
(here the topology is given by the the completion of $H_{[P],\sigma}^{\infty
}\otimes\left(  H_{[P],\sigma}^{^{\prime}}\right)  ^{\infty}$ using the
semi-norms $\mu\otimes\eta$ for $\mu$ and $\eta$ respectively continuous
seminorms on $H_{[P],\sigma}^{\infty}$ and $\left(  H_{[P],\sigma}^{^{\prime}%
}\right)  ^{\infty}$ and
\[
\mu\otimes\eta(z)=\inf\sum\mu(x_{i})\eta(y_{i})
\]
the $\inf$ is over all expressions $z=\sum x_{i}\otimes y_{i}$). For each
$P\in\lbrack P]$, $\nu\in\left(  \mathfrak{a}_{[P]}\right)  _{\mathbb{C}}$ we
have an action of $G\times G$ on $HS(H_{[P],\sigma})$ given by
\[
\Pi_{\bar{P},\sigma,\nu}(g_{1},g_{2})T=\pi_{\bar{P},\sigma,\nu}(g_{1}%
)T\pi_{\bar{P}},_{\sigma,\nu}(g_{2})^{-1}.
\]
The $G\times G-C^{\infty}$ vectors are the same as the $K\times K-C^{\infty}$
vectors (see \cite{CWParam} Proposition A.3). If $(\pi,V)$ is a smooth
representation of $G$ on a Fr\'{e}chet space and if $\chi$ is a unitary
character of $N_{o}$ then set
\[
V_{\chi}=\text{Closure}(\text{Span}\left\{  (\pi(n)-\chi(n))v|n\in N_{o},v\in
V\right\}  ).
\]
If we act on on $HS(H_{[P],\sigma})^{\infty}$ by $\pi_{\bar{P},\sigma,\nu
}\otimes I$ we have
\[
HS(H_{[P],\sigma})_{\bar{P},\nu,\chi}^{\infty}=\left(  H_{[P],\sigma}^{\infty
}\right)  _{\bar{P},\nu,\chi}\otimes_{F}\left(  H_{[P],\sigma}^{^{\prime}%
}\right)  ^{\infty}.
\]
If we form the Fr\'{e}chet space quotient we have%
\[
HS(H_{[P],\sigma})^{\infty}/HS(H_{[P],\sigma})_{\bar{P},\nu\chi}^{\infty
}=\left(  H_{[P],\sigma}^{\infty}/\left(  H_{[P],\sigma}^{\infty}\right)
_{\bar{P},\nu\chi}\right)  \otimes_{F}\left(  H_{[P],\sigma}^{^{\prime}%
}\right)  ^{\infty}%
\]
and we can drop the $F$ in the tensor product since $H_{[P],\sigma}^{\infty
}/\left(  H_{[P],\sigma}^{\infty}\right)  _{\bar{P},\nu\chi}$ is finite
dimensional (this follows from Theorem 54 in \cite{WPT} combined with Theorem
\ref{Isomorphism}). If $\nu\in\left(  \mathfrak{a}_{P}^{\ast}\right)
_{\mathbb{C}}$ is such that $A_{k}(\bar{P},\sigma,\nu)$ is invertible and if
$k\in N_{K}(A_{P})$ then
\[
C_{k}(\bar{P},\sigma,\nu)(HS(H_{[P],\sigma})_{\bar{P},\nu,\chi}^{\infty
})\subset HS(H_{[P],k\sigma})_{\bar{P},k\nu,\chi}^{\infty}.
\]
Thus $C_{k}(\bar{P},\sigma,\nu)$ induces a transformation
\[
\tilde{C}_{k,\chi}(\bar{P},\sigma,\nu):HS(H_{[P],\sigma})^{\infty
}/HS(H_{[P],\sigma})_{\bar{P},\nu\chi}^{\infty}\rightarrow HS(H_{[P],k\sigma
})^{\infty}/HS(H_{[P],k\sigma})_{\bar{P},k\nu\chi}^{\infty}.
\]
If $\nu\in\mathfrak{a}_{P}^{\ast}$ define a pairing $(...|...)_{\nu}$ between
$Wh_{\chi_{|N_{o}\cap M_{P}}^{-1}}(H_{\sigma}^{\infty})$ and $H_{[P],\sigma
}^{\infty}/\left(  H_{[P],\sigma}^{\infty}\right)  _{\bar{P},i\nu\chi}$ by
\[
(\lambda|\tilde{u})_{\nu}=J_{\chi^{-1}}(P,\sigma,i\nu)(\lambda)(u)
\]
here $u\in\tilde{u}=u+\left(  H_{[P],\sigma}^{\infty}\right)  _{\bar{P}%
,\nu,\chi}$. Recall that if $(\pi,H)$ is a unitary representation of $G$ then
$(\bar{\pi},\bar{H})$ is the unitary representation on the complex conjugate
space (i.e. the $\mathbb{C}$--vector space $H$ with $z\in\mathbb{C} $ acting
by $\bar{z}$ and $\bar{\pi}(g)=\pi(g)).$ Let $(\pi^{\prime},H^{\prime})$
denote the continuous dual representation of $H$ with $\pi^{\prime}%
(g)\lambda=\lambda\circ\pi(g)^{-1}$ and $(\pi^{\prime},H^{\prime})$ is
unitarily equivalent with $(\bar{\pi},\bar{H})$.

\begin{lemma}
If $\nu\in\mathfrak{a}_{P}^{\ast}$ then $(...|...)_{\nu}$ defines a
non-degenerate, bilinear pairing between $Wh_{\chi_{|N_{o}\cap M_{P}}%
}(H_{\sigma}^{\infty})$ and $H_{[P],\sigma}^{\infty}/\left(  H_{[P],\sigma
}^{\infty}\right)  _{\bar{P},i\nu,\chi}$ inducing a linear bijection
\[
\tau_{P,\sigma,\nu,\chi}:H_{[P],\sigma}^{\infty}/\left(  H_{[P],\sigma
}^{\infty}\right)  _{\bar{P},i\nu,\chi}\rightarrow Wh_{\chi_{|N_{o}\cap M_{P}%
}}(H_{\sigma}^{\infty})^{\ast}%
\]
composing this with the inverse to $\lambda\mapsto\lambda^{\#}$ we have an
isomorphism
\[
\tau_{P,\sigma,\nu,\chi}^{\#}:H_{[P],\sigma}^{\infty}/\left(  H_{[P],\sigma
}^{\infty}\right)  _{\bar{P},i\nu,\chi}\rightarrow\overline{Wh_{\chi
_{|N_{o}\cap M_{P}}}(H_{\sigma}^{\infty})}.
\]

\end{lemma}

Note that if $V$ is a Fr\'{e}chet space and if $\lambda\in V^{\prime}$ then
defining $\bar{\lambda}(u)=\overline{\lambda(u)}$ then $\bar{\lambda}\in
\bar{V}^{\prime}$. This implies

\begin{lemma}
If $(\pi,H_{\pi})$ is a unitary representation of $G$ and if $\lambda\in
Wh_{\chi^{-1}}(\bar{H}_{\bar{\pi}}^{\infty})$ then $\bar{\lambda}\in Wh_{\chi
}(H_{\pi}^{\infty})$.
\end{lemma}

If $T\in\mathcal{S}(H_{[P],\sigma})$ and $\nu\in\mathfrak{a}_{P}^{\ast}$ then
define $\Psi_{P,\sigma,\nu,\chi}(T)=\left(  \tau_{P,\sigma,\nu,\chi}%
^{\#}\otimes I\right)  (T)\in\overline{Wh_{\chi_{|N_{o}\cap M_{P}}}(H_{\sigma
}^{\infty})}\otimes\bar{H}_{[P],\sigma}^{\infty}$. Here we have used the
identification of $\mathcal{S}(H_{[P],\sigma})$ with $H_{[P],\sigma}^{\infty
}\otimes_{F}\bar{H}_{[P],\sigma}^{\infty}$.

For later developments we give in $\Psi_{P,\mu,\nu,\chi}(T)$ in coordinates
and leave it to the reader to unwind the definitions used in the definition of
$\Psi_{P,\mu,\nu,\chi}(T)$ especially the pairing $(...,...)_{\nu}$.

\begin{lemma}
\label{explicit}Let $\lambda_{i,\mu}$ be an orthonormal basis of
$Wh_{\chi_{|N_{o}\cap M_{P}}}(H_{\mu}^{\infty})$ and let $\left\{  v_{i,\mu
}\right\}  $ be an orthonormal basis of $H_{[P],\mu}$ such that for each $i$
there exists $\gamma_{i}\in\hat{K}$ such that $v_{i,\mu}\in H_{[P],\mu}%
[\gamma_{i}]$ then%
\[
\Psi_{P,\mu,\nu,\chi}(T)=\sum_{i,j}J_{\chi^{-1}}(P,\mu,i\nu)(\lambda_{i,\mu
})(Tv_{j,\mu})\lambda_{i,\mu}^{\#}\otimes v_{j,\mu}^{\#}%
\]
if $T\in\mathcal{S}(H_{[P],\mu})$.
\end{lemma}

\begin{theorem}
If $\nu\in\mathfrak{a}_{P}^{\ast}$ consider the map%
\[
C_{k}(\bar{P},\sigma,i\nu):\mathcal{S}(H_{[P],\sigma})\rightarrow
\mathcal{S}(H_{[P],k\sigma})
\]
than there exists a unique map
\[
CW_{k,\chi}(P,\sigma,\nu):\overline{Wh_{\chi_{|N_{o}\cap M_{P}}}(H_{\sigma
}^{\infty})}\otimes\bar{H}_{[P],\sigma}^{\infty}\rightarrow\overline
{Wh_{\chi_{|N_{o}\cap M_{P}}}(H_{k\sigma}^{\infty})}\otimes\bar{H}%
_{[P],k\sigma}^{\infty}%
\]
such that if $T\in\mathcal{S}(H_{[P],\sigma})$ then%
\[
CW_{k,\chi}(P,\sigma,\nu)(\Psi_{P,\sigma,\nu,\chi}(T))=\Psi_{P,k\sigma
,k\nu,\chi}(C_{k}(\bar{P},\sigma,i\nu)(T)).
\]
Furthermore, $CW_{k,\chi}(P,\sigma\,\nu)$ extends to a unitary operator on
$\overline{Wh_{\chi_{|N_{o}\cap M_{P}}}(H_{\sigma}^{\infty})}\otimes\bar
{H}_{[P],\sigma}$ and if $\eta\in$ $\overline{Wh_{\chi_{|N_{o}\cap M_{P}}%
}(H_{\sigma}^{\infty})}\otimes\bar{H}_{[P],\sigma}$
\[
\nu\mapsto CW_{k,\chi}(P,\sigma,\nu)(\eta)
\]
extends to a real analytic map to $\overline{Wh_{\chi_{|N_{o}\cap M_{P}}%
}(H_{k\sigma}^{\infty})}\otimes\bar{H}_{[P],k\sigma}$.
\end{theorem}

\begin{proof}
Take
\[
CW_{k,\chi}(P,\sigma,\nu)=\tau_{P,k\sigma,k\nu,\chi}^{\#}\circ\tilde
{C}_{k,\chi}(\bar{P},\sigma,\nu)\circ\left(  \tau_{P,\sigma,\nu}^{\#}\right)
^{-1}%
\]
where defined. Writing out $\Psi_{P,k\sigma,\nu,\chi}$ in terms of $J_{\chi
}(P,k\sigma,i\nu)$ using Lemma \ref{explicit} one sees that the real
analyticity follows from that of $C_{k,\chi}(\bar{P},\sigma,i\nu)(T)$ (see
Theorem \ref{maass-selberg}). The unitarity is a direct consequence of Lemma
9.24.8 p. 273 in part III of \cite{HCV} by relating our $CW_{k,\chi}%
(P,\sigma\,\nu)$ to Harish-Chandra's {}$^{o}C$ function in \cite{HCV}.
\end{proof}

We will now describe the analogues of the bundles in Section
\ref{G-Plancherel Space} in the Whittaker case. Let $(\sigma,H_{\sigma})$ be
an irreducible square integrable representation of ${}^{o}M_{P}$. The bundle
is an $N_{K}(A_{P})$--bundle over $N_{K}(A_{P})\sigma\times\mathfrak{a}%
_{P}^{\ast}$ with fiber over $(k\sigma,\nu)$ equal to $\overline
{Wh_{\chi_{|N_{o}\cap M_{P}}^{-1}}(H_{\sigma}^{\infty})}\otimes\bar
{H}_{[P],\sigma}$. The action of $k$ is
\[
k\xi=CW_{k_{1},\chi}(P,k\sigma,\nu)(\xi)
\]
for $\xi\in\overline{Wh_{\chi_{|N_{o}\cap M_{P}}^{-1}}(H_{k\sigma}^{\infty}%
)}\otimes\bar{H}_{[P],k\sigma}$. We denote the bundle by $\mathbb{W}_{\chi
}(P,\sigma)$. A continuous $N_{K}(A_{P})$--invariant cross-section of
$\mathbb{W}_{\chi}(P,\sigma)$ is a map
\[
\alpha:N_{K}(A_{P})\sigma\times\mathfrak{a}_{P}^{\ast}\rightarrow\cup_{k\in
N_{K}(A_{P}),\nu\in\mathfrak{a}_{P}^{\ast}}\overline{Wh_{\chi_{|N_{o}\cap
M_{P}}^{-1}}(H_{k\sigma}^{\infty})}\otimes\bar{H}_{[P],k\sigma}%
\]
such that%
\[
\alpha(k\mu,k\nu)=CW_{k,\chi}(P,\mu,\nu)\left(  \alpha(\mu,\nu)\right)
\]
and $\nu\mapsto\alpha(\mu,\nu)$ is continuous.

The following is clear from the definitions and the results in Section
\ref{G-Plancherel Space}.

\begin{proposition}
If $\varphi\in\mathcal{C}(G)$ then $(\mu,\nu)\mapsto\Psi_{P,\mu,\nu,\chi}%
(\pi_{\bar{P},\mu,i\nu}(\varphi))$ defines a continuous cross-section of
$\mathbb{W}_{\chi}(P,\sigma)$.
\end{proposition}

Note that the bundle $\mathbb{W}_{\chi}(P,\sigma)$ is also a Hilbert
$G$-vector bundle with action $g\xi=\left(  I\otimes\bar{\pi}_{\bar{P}%
,\mu,i\nu}(g)\right)  \xi$ for $\xi\in\mathbb{W}_{\chi}(P,\sigma)_{(\mu,\nu)}%
$. This action of $G$ commutes with the action of $N_{K}(A_{P})$.

We also define the Fr\'{e}chet $N_{K}(A_{P})\times G$--vector bundle
$\mathbb{W}_{\chi}^{\infty}(P,\sigma)$ with fiber $\overline{Wh_{\chi
_{|N_{o}\cap M_{P}}^{-1}}(H_{k\sigma}^{\infty})}\otimes\bar{H}_{[P],k\sigma
}^{\infty}$ over $(k\sigma,\nu)$ We define a subspace of $\Gamma
(\mathbb{W}_{\chi}^{\infty}(P,\sigma))$, analogous to $\Gamma_{rd}%
(\mathcal{\underline{\mathcal{S}}}(HS([P],\sigma)))$ of Section
\ref{G-Plancherel Space}, $\Gamma_{rd}(\mathbb{W}_{\chi}^{\infty}(P,\sigma)),
$ consisting of the $\alpha\in\Gamma(\mathbb{W}_{\chi}^{\infty}(P,\sigma))$
such that for all $d,r\in\mathbb{Z}_{\geq0}$%
\[
\eta_{d,r}(\alpha)=\mathrm{sup}\{\left\Vert \nu\right\Vert )^{d}\left\Vert
\bar{\pi}_{k\sigma}(C_{K})^{r}\alpha(k\sigma,\nu)\right\Vert |k\in N_{K}%
(A_{P}),\nu\in\mathfrak{a}_{P}^{\ast}\}<\infty.
\]

Theorem \ref{J-estimate} says that if $\chi$ is generic then there exist,
$d\in\mathbb{Z}$ and $q$ a continuous seminorm on $H_{[P],\sigma}^{\infty}$
such that
\[
\left\vert J_{\chi}(P,\sigma,i\nu)(\lambda)(\pi_{\bar{P},\sigma,i\nu
}(g)u)\right\vert \leq q(u)a^{\rho_{o}}(1+\log\left\Vert a\right\Vert
)^{d}\left\Vert \lambda\right\Vert _{\sigma}%
\]
here the norm on $\lambda$ is the Hilbert norm coming form $(...,...)_{\sigma
}$. This implies that for each $g$ we can form a cross-section of the bundle
$\mathbb{W}_{\chi}(P,\sigma)$ given by%
\[
\Omega_{\chi}(\mu,\nu)(g)=\sum_{i,j}J_{\chi^{-1}}(P,\mu,i\nu)(\lambda_{i,\mu
})(\pi_{\bar{P},\mu,i\nu}(g)v_{j,\mu})\lambda_{i,\mu}^{\#}\otimes v_{j,\mu
}^{\#}%
\]
for $\mu\in N_{K}(A_{P})\sigma$.

Define for $\varphi\in\mathcal{C}(G)$, $\varphi_{\chi}(g)=\int_{N_{o}}%
\chi(n)^{-1}f(ng)dn$. The following is a restatement of Theorem 58 and
Corollary 62 in \cite{WPT}

\begin{theorem}
\label{first version}Let $T_{[P],[\sigma]}$ be the projection of
$\mathcal{C(}G)$ onto $\mathcal{C}(G)_{[P],[\sigma]}=L^{2}(G)_{[P],\sigma
}^{\infty}\cap\mathcal{C}(G)$ relative to the direct sum decomposition in
Theorem \ref{smoothdecomp} If $\varphi\in\mathcal{C}(G)$ is $K\times
K$--finite then $T_{[P].\sigma}(\varphi)_{\chi}(g)$ is equal to
\[
\frac{1}{d(\sigma)}\int_{N_{K}(A_{P})\times\mathfrak{a}_{P}^{\ast}%
}\left\langle \Psi_{P,\mu,\nu,\chi}(\pi_{\bar{P},k\sigma,ik\nu}(\varphi
)),\Omega_{\chi}(k\sigma,k\nu)(g)\right\rangle \mu(\sigma,i\nu)dkd\nu.
\]
and if $\varphi_{1},\varphi_{2}\in\mathcal{C}(G)_{[P],\sigma}$ are $K\times K
$ finite then%
\[
\int_{N_{o}\backslash G}\left\langle \left(  \varphi_{1}\right)  _{\chi
},\left(  \varphi_{2}\right)  _{\chi}\right\rangle dg=
\]%
\[
\frac{1}{d(\sigma)}\int_{N_{K}(A_{P})\times\mathfrak{a}_{P}^{\ast}%
}\left\langle \Psi_{P,\mu,\nu,\chi}(\pi_{\bar{P},k\sigma,ik\nu}(\varphi
_{1})),\Psi_{P,\mu,\nu,}(\pi_{\bar{P},k\sigma,ik\nu}(\varphi_{2}%
))\right\rangle \mu(\sigma,i\nu)dkd\nu.
\]

\end{theorem}

Let $L^{2}(\mathbb{W}(P,\sigma))$ denote the Hilbert space completion of the
space of all $f\in\Gamma_{rd}(\mathbb{W}^{\infty}(P,\sigma))$ such that%
\[
\int_{N_{K}(\mathfrak{a}_{P}^{\ast})\times\mathfrak{a}_{P}^{\ast}}\left\langle
f(k\sigma,\nu),f(k\sigma,\nu)\right\rangle dk\mu(\sigma,i\nu)d\nu<\infty.
\]

Define, for $\nu\in\mathfrak{a}_{P}^{\ast}$ and $\mu\in N_{K}(\mathfrak{a}%
_{P})\sigma,$ and $\varphi\in\mathcal{C}(G)$, $S_{[P],\sigma,\chi}%
(\varphi)(\mu,\nu)=\Psi_{P,\mu,\nu,\chi}(\pi_{\bar{P},\mu,i\nu}(\varphi))$.
Then $\Psi_{P,\mu,\nu,\chi}(\pi_{\bar{P},\mu,i\nu}(\varphi))\in\Gamma
_{rd}(\mathbb{W}^{\infty}(P,\sigma))^{N_{K}(A_{P})}. $

\begin{theorem}
\label{second-version}The map%
\[
S_{[P],\sigma,\chi}:\mathcal{C}(G)_{[P],\sigma}\rightarrow\Gamma
_{rd}(\mathbb{W}^{\infty}(P,\sigma))^{N_{K}(A_{P})}%
\]
extends to a surjection of \thinspace$L^{2}(G)_{[P],\sigma}$ onto $\Gamma
^{2}(\mathbb{W}_{\chi}(P,\sigma))^{N_{K}(\mathfrak{a}_{P})}$.
\end{theorem}

\begin{proof}
The previous Theorem implies that the map is unitary thus it is enough to show
that the image is dense. We write out the map $S_{[P],\sigma,\chi}$ \ using
Lemma \ref{explicit} as%
\[
S_{[P],\sigma,\chi}(\varphi)=\sum_{i,j}J_{\chi^{-1}}(P,\mu,i\nu)(\lambda
_{i,\mu})(\pi_{\bar{P},\mu,i\nu}(\varphi)v_{j,\mu})\lambda_{i,\mu}^{\#}\otimes
v_{j,\mu}^{\#}.
\]
Let $F\subset\hat{K}$ be a finite set. We now assume that $\varphi$ is
$K\times K$--finite and transforms by elements of $F$ on the left and right.
Then%
\[
\sum_{\gamma_{j}\in F}J_{\chi^{-1}}(P,\mu,i\nu)(\lambda_{i,\mu})(\pi_{\bar
{P},\mu,i\nu}(\varphi)v_{j,\mu})\lambda_{i,\mu}^{\#}\otimes v_{j,\mu}^{\#}.
\]
We have seen that if $\beta\in\mathcal{H}(P,\sigma))$ then $S(T(\beta))=\beta$
(see the proof of Theorem \ref{density}). This and Lemma \ref{surjective}
implies that the image contains all such $K\times K$ finite compactly
supported smooth elements of $\Gamma(\mathbb{W}_{\chi}(P,\sigma))^{N_{K}%
(A_{P})}$.
\end{proof}

\section{The Hilbert bundle version of the Whittaker Plancherel Theorem}

Let $M$ be a real reductive group and $N$ the unipotent radical of a minimal
parabolic subgroup. If $\chi$ is a generic character of $N$ we denote by
$\mathcal{E}_{2,\chi}(M)$ the set of square integrable integrable
representations of $M$, $(\sigma,H_{\sigma})$, such that $Wh_{\chi^{-1}%
}(H_{\sigma}^{\infty})\neq0$. We will use the notation of Theorem
\ref{Abs-Planch} modified using the equivalence relation in the statement on
$\mathcal{E}_{2,\chi}(M)$ the set of classes will be denoted $\overline
{\mathcal{E}}_{2,\chi}(M)$.

Fix, $\chi$, a generic unitary character of $N_{o}$ and $P$ a standard
parabolic subgroup of $G$. Let for $[\sigma]\in\overline{\mathcal{E}}%
_{2,\chi_{|N_{o}\cap M_{P}}}(^{o}M_{P})$, $\mathcal{C(}N_{o}\backslash
G,\chi)_{[P],\sigma}$ be the completion of the space $\mathcal{C}%
_{[P],[\sigma\}}(G)_{\chi}=\{f_{\chi}|f\in\mathcal{C}_{[P],[\sigma\}}(G)\}$
here is as in Section \ref{G-Plancherel Space}. In \cite{WPT} Theorem 34 we proved

\begin{theorem}
Assume that $\chi$ is generic then $\mathcal{C(}N_{o}\backslash G,\chi)$ is
the closure in $\mathcal{C(}N_{o}\backslash G,\chi)$ of the orthogonal direct
sum%
\[
\oplus_{\lbrack P]\in\mathcal{P}(G),\left[  \sigma\right]  \in\overline
{\mathcal{E}}_{2,\chi}({}^{o}M_{P})}\mathcal{C(}N_{o}\backslash G,\chi
)_{[P],\sigma}.
\]

\end{theorem}

Let $P$ be a standard parabolic subgroup of $G$ and Let $\sigma$ be an
irreducible, square integrable representation of ${}^{o}M_{P}$. Let $\left\{
\lambda_{i,\sigma}\right\}  $ and $\left\{  v_{i,\sigma}\right\}  $ be
respectively orthonormal bases of $Wh_{\chi_{|N_{o}\cap M_{P}}^{-1}}%
(H_{\sigma}^{\infty})$ and $H_{[P],\sigma}$ as in the previous section. If
$\varphi\in\mathcal{C(}N_{o}\backslash G,\chi),\nu\in\mathfrak{a}_{P}^{\ast}$
set%
\[
a(\varphi,\sigma,\nu)_{ij}=\int_{N_{o}\backslash G}J_{\chi^{-1}}(P,\sigma
,i\nu)(\lambda_{i,\sigma})(\pi_{\bar{P},\sigma,i\nu}(g)v_{i,\sigma}%
)\varphi(g)d\bar{g}.
\]
Then arguing as in Proposition \ref{Weyl-ests} using $R_{(I+C_{K})^{r}}$ one
proves that there exists for each $r\in\mathbb{Z}_{>0}$ a continuous
semi-norm, $q_{\sigma,r}$ on $\mathcal{C(}N_{o}\backslash G,\chi)$ such that
\[
\left\vert a(\varphi,\sigma,\nu)_{ij}\right\vert \leq q_{\sigma,r}%
(\varphi)(1+\lambda_{\gamma_{j}})^{-r}\text{.}%
\]
This allows us to define for $\varphi\in\mathcal{C(}N_{o}\backslash G,\chi)$%
\[
SW_{\left[  P\right]  ,\sigma,\chi}(\varphi)(\mu,\nu)=\sum_{i,j}a(\varphi
,\mu,\nu)_{ij}\lambda_{i,\mu}^{\#}\otimes v_{i,\mu}^{\#}%
\]
for $\mu\in N_{K}(A_{P})\sigma$ and $\nu\in\mathfrak{a}_{P}^{\ast}$. Also
Lemma \ref{Basic} implies that if $\varphi=f_{\chi}$ for $f\in\mathcal{C}(G)$
then
\[
SW_{\left[  P\right]  ,\sigma,\chi}(\varphi)=S_{[P],\sigma,\chi}(f)\in
\Gamma^{2}(\mathbb{W}_{\chi}(P,\sigma))^{N_{K}(A_{P})}.
\]
Thus we have the main result

\begin{theorem}
\label{BundleFormWhitt}The map
\[
SW_{\left[  P\right]  ,\sigma,\chi}:\mathcal{C(}N_{o}\backslash G,\chi
)_{[P]},_{\sigma}\rightarrow{}_{rd}\Gamma^{2}(\mathbb{W}_{\chi}^{\infty
}(P,\sigma))^{N_{K}(A_{P})}%
\]
extends to a unitary $G$--isomorphism of the right regular action of $G$ on
$L^{2}\mathcal{(}N_{o}\backslash G,\chi)_{[P]},_{\sigma}$ (the closure of
$\mathcal{C(}N_{o}\backslash G,\chi)_{[P]},_{\sigma}$ in $L^{2}\mathcal{(}%
N_{o}\backslash G,\chi)$) with the space $\Gamma^{2}(\mathbb{W}_{\chi
}(P,\sigma)$)$^{N_{K}(A_{P})}$ with inverse given by%
\[
\alpha\mapsto\frac{1}{d(\sigma)}\int_{N_{K}(A_{P})\times\mathfrak{a}_{P}%
^{\ast}}\left\langle \alpha(k\sigma,k\nu),\Omega_{\chi}(k\sigma,k\nu
)(\cdot)\right\rangle \mu(k\sigma,i\nu)dkd\nu
\]
for $\alpha\in$ ${}_{rd}\Gamma^{2}(\mathbb{W}_{\chi}^{\infty}(P,\sigma
))^{N_{K}(A_{P})}$.
\end{theorem}

\begin{proof}
Combine Theorems \ref{first version} and \ref{second-version}.
\end{proof}

\begin{corollary}
As a representation of $G$, $L^{2}(N_{o}\backslash G,\chi)$ is equivalent
with
\[
\oplus_{\lbrack P]\in\mathcal{P}(G),\left[  \sigma\right]  \in\overline
{\mathcal{E}}_{2,\chi}({}^{o}M_{P})}\Gamma^{2}(\mathbb{W}_{\chi}%
(P,\sigma))^{N_{K}(A_{P})}.
\]

\end{corollary}

\section{The direct integral version of the Plancherel
Theorem\label{Direct-integral}}

The purpose of this section is to show how the results on the Plancherel Space
for the Harish-Chandra and Whittaker cases imply the corresponding direct
integral decompositions. We first recall the parametrization of an open full
measure subset of the tempered dual, $\mathcal{E}_{temp}(G)$, of the real
reductive group $G$.

We now recall some notation from Theorem \ref{Abs-Planch}. Let $P$ be a
cuspidal, standard, parabolic subgroup of $G$. If $\omega\in\mathcal{E}_{2}%
({}^{o}M_{P})$ and if $\sigma_{\omega}\in\omega$ and if $s\in W(A_{P})$ then
define $s\omega$ to be the equivalence class of $k\sigma_{\omega}$ for $k\in
K$ and $Ad(k)_{|\mathfrak{a}_{P}}=s$. Let $[\omega]$ denote $W(A_{P})\omega$
on $\mathcal{E}_{2}({}^{o}M_{P})$ and denote the set of such classes as
$\overline{\mathcal{E}_{2}}({}^{o}M_{P})$. Let $\mathcal{P}_{cusp} $ denote
the set of associativity classes of cuspidal standard parabolic subgroups of
$G$. If $[P]\in\mathcal{P}_{cusp}$ then $A_{P,}M_{P}$ and {}$^{o}M_{P}$ depend
only on $[P]$. Since $\mathcal{E}_{2}({}^{o}M_{P})$ has been endowed the
discrete topology we may choose a cross-section to that action of $W(A_{P})$,
$\mathcal{E}_{2}({}^{o}M_{P})^{o}.$ As in Lemma \ref{FundDom} let
$\Omega_{\lbrack P]}^{\prime}=\Omega_{W(A_{P})}\cap\left(  \mathfrak{a}%
_{P}^{\ast}\right)  ^{\prime}\subset\mathfrak{a}_{P}^{\ast}$ and
$U_{[P]}^{\prime}=U_{W(A_{P})}\cap\left(  \mathfrak{a}_{P}^{\ast}\right)
^{\prime}$. Let
\[
\mathcal{A}(G,[P])^{\prime}=\cup_{\omega\in\mathcal{E}_{2}({}^{o}M_{P})^{o}%
}\sigma_{\omega}\times U_{[P]}^{\prime}.
\]
If $s,t\in W(A_{P})$ and $\omega\in\mathcal{E}_{2}({}^{o}M_{P})^{o}$ then
$s(\sigma_{\omega}\times U_{[P]}^{\prime})\cap t(\sigma_{\omega}\times
U_{[P]}^{\prime})\neq\emptyset$ then $s=t$. If $(\pi,H)$ is an unitary
representation of $G$ then its equivalence class will be denoted by $[\pi]$.
Set
\[
\Phi([P],\omega,\nu)=[\pi_{P,\sigma_{\omega},i\nu}]
\]
for $[P]\in\mathcal{P}_{cusp}$, $\omega\in\mathcal{E}_{2}({}^{o}M_{P}),\nu
\in\mathfrak{a}_{P}^{\ast}$. Theorem 14.12.3 in \cite{RRGII} implies

\begin{theorem}
The map $\Phi$ of $\cup_{\lbrack P]\in\mathcal{P}_{cusp}}\mathcal{A}%
(G,[P])^{\prime}$ with the disjoint union topology to $\mathcal{E}_{temp}(G) $
is a homeomorphism onto an open subset of $\mathcal{E}_{temp}(G) $ that is of
co--measure $0$.
\end{theorem}

We now move to the situation and notation of Section \ref{G-Plancherel Space}.
We note that if $\alpha,\beta\in\Gamma^{2}(\mathbf{HS}(H_{[P],\sigma}%
))^{N_{K}(A_{P})}$ and if $\alpha_{|\mathcal{A}(G,[P])^{\prime}}%
=\beta_{|\mathcal{A}(G,[P])^{\prime}}$ then $\alpha=\beta$. Also
$\mathbf{HS}(H_{[P],\sigma})_{|\mathcal{A}(G,[P])^{\prime}}$ is the trivial
vector bundle $\mathcal{A}(G,[P])^{\prime}\times HS(H_{[P],\sigma})$ and
$G\times G$--action given by $(g,h).(([P],\sigma,\nu),T)=(([P],\sigma,\nu
),\Pi_{P,,\sigma,i\nu}(g,h)T)$ now Theorem \ref{Abs-Planch} yields the
standard direct integral version of the Harish-Chandra Plancherel Theorem (see
Theorem 14,12,4 in \cite{RRGII}).

We will now do the same for the Whittaker Plancherel Theorem. For this we make
the (slight) change of notation
\[
\Phi([P],\omega,\nu)=[\pi_{P,\sigma_{\omega},i\nu}].
\]
This doesn't effect the above theorem. As above we have

\begin{lemma}
Let $\alpha,\beta\in\Gamma(\mathbb{W}_{\chi}(P,\sigma))^{N_{K}(A_{P})}$ and if
$\alpha_{|\sigma\times U_{[P]}^{\prime}}=\beta_{|\sigma\times U_{[P]}^{\prime
}}$ then $\alpha=\beta$.
\end{lemma}

Note that if $\alpha,\beta$ square integrable elements of $\Gamma
(\mathbb{W}_{\chi}(P,\sigma))^{N_{K}(A_{P})}$ then%
\[
\int_{U_{[P]}^{\prime}}\left\langle \alpha(\sigma,\nu),\beta(\alpha
,\nu)\right\rangle \mu(\sigma,i\nu)d\nu=\int_{N_{K}(A_{P})\times U_{P}%
^{\prime}}\left\langle \alpha(k\sigma,k\nu),\beta(k\alpha,k\nu)\right\rangle
\mu(\sigma,i\nu)d\nu.
\]

This implies that the unitary representation $\Gamma^{2}(\mathbb{W}_{\chi
}(P,\sigma))^{N_{K}(A_{P})}$ is equivalent with%
\[
\overline{Wh_{\chi_{|N_{o}\cap M_{P}}^{-1}}(H_{\sigma}^{\infty})}\otimes
\int_{U_{P}^{\prime}}(\bar{\pi}_{\bar{P},\sigma,i\nu,}\bar{H}_{[P],\sigma}%
)\mu(\sigma,i\nu)d\nu.
\]
We note that $\bar{\pi}_{\bar{P},\sigma,i\nu,}$ is equivalent with $\pi
_{\bar{P},\bar{\sigma},-\iota\nu}$ and the map $Wh_{\chi_{|N_{o}\cap M_{P}}%
}(H_{\bar{\sigma}\ }^{\infty})$ to $\overline{Wh_{\chi_{|N_{o}\cap M_{P}}%
^{-1}}(H_{\sigma}^{\infty})}$ given by $\lambda\mapsto\bar{\lambda} $ defines
a unitary isomorphism. Thus $\Gamma^{2}(\mathbb{W}_{\chi}(P,\sigma
))^{N_{K}(A_{P})}$ is equivalent with%
\[
Wh_{\chi_{|N_{o}\cap M_{P}}}(H_{\bar{\sigma}\ }^{\infty})\otimes\int
_{U_{P}^{\prime}}(\pi_{\bar{P},\bar{\sigma},-i\nu,}\bar{H}_{[P],\sigma}%
)\mu(\sigma,i\nu)d\nu.
\]
This and Lemma \ref{Invariance} yields the form of the direct integral
decomposition for $L^{2}(N_{o}\backslash G,\chi)$ in Theorem 15.9.1 of
\cite{RRGII}

\begin{theorem}
As a representation of $G$,$L^{2}(N_{o}\backslash G,\chi)$ is equivalent with%
\[
\bigoplus_{\lbrack P]\in\mathcal{P}_{cusp},[\omega]\in\mathcal{E}_{2}%
,\chi_{|N_{o}\cap M_{P}}({}^{o}M_{P})}\int_{\mathfrak{a}_{P}^{\ast}}%
Wh_{\chi_{|N_{o}\cap M_{P}}}(H_{\sigma\ }^{\infty})\otimes(\pi_{\bar{P}%
,\sigma,i\nu,}H_{[P],\sigma})\mu(\sigma,i\nu)d\nu.
\]

\end{theorem}

\appendix

\section{Some simple results}

The following lemmas are no doubt well known .

\begin{lemma}
\label{FundDom}Let $S$ be a finite subgroup of $O(n)$. Then there exists an
$S$--invariant subset of $\mathbb{R}^{n}$, $\Omega_{S}$, that is the
complement of a finite union of hyperplanes and an open subset $U_{S}%
\subset\Omega_{S}$ such that if $s\in S-\{I\}sU_{S}\cap U_{S}=\emptyset$ and
$U_{s\in S}sU_{S}=\Omega_{S}$.
\end{lemma}

\begin{proof}
If $s\in S$ set
\[
\left(  \mathbb{R}^{n}\right)  ^{s}=\{v\in\mathbb{R}^{n}|sv=v\}.
\]
If $s\neq I$ then $\left(  \mathbb{R}^{n}\right)  ^{s}$ is a proper subspace.
If $\lambda_{o}\notin\cup_{s\in S-\{I\}}\left(  \mathbb{R}^{n}\right)  ^{s}$
then, obviously, $s\lambda_{o}\neq\lambda_{o}$ for all $s\in S-\{I\}$. Fix
such a $\lambda_{o}$. We now do a standard construction. Let
\[
\Omega_{S}=\{\lambda|\left\langle \lambda,s\lambda_{o}-t\lambda_{o}%
\right\rangle \neq0,\mathrm{\ for\ }s\neq t,s,t\in S\}.
\]
Then $s\Omega_{S}=\Omega_{S}$ for $s\in S$. If $\lambda\in\Omega_{S}$,$s,t\in
S,$ and if $\left\vert s\lambda-\lambda_{o}\right\vert =\left\vert
t\lambda-\lambda_{o}\right\vert $ then
\[
\left\langle \lambda,s^{-1}\lambda_{o}-t^{-1}\lambda_{o}\right\rangle =0
\]
which implies $s=t$. Set
\[
U_{S}=\{\lambda\in\Omega_{S}|\left\vert \lambda-\lambda_{o}\right\vert
<\left\vert s\lambda-\lambda_{o}\right\vert ,s\in S-I\}.
\]
Then $U_{S}$ and $\Omega_{S}$ have the desired properties.
\end{proof}

The next result is true for any unimodular, locally compact group we will
prove it only for the real reductive group, $G$. If we fix the Haar measure on
$G$ \ Then associated with it is the Plancherel Measure on, $\mathcal{E}(G)$,
the set of equivalence classes of irreducible unitary representations of $G$.
If $\omega\in\mathcal{E}(G)$ then let $\omega^{\ast}$ be the equivalence class
of the dual unitary representation of any element of the class of $\omega$. If
$S\subset\mathcal{E}(G)$ is a $\mu$--measurable subset then set $S^{\ast
}=\{\omega^{\ast}|\omega\in S\}$. If $\omega\in\mathcal{E}(G)$ fix a choice of
$(\pi_{\omega},H_{\omega})\in\mathcal{E}(G)$.

\begin{lemma}
\label{Invariance}If $S\subset\mathcal{E}(G)$ is a $\mu$--measurable subset
then $\mu(S)=\mu(S^{\ast})$.
\end{lemma}

\begin{proof}
The Plancherel Theorem implies that $\mu$ is the unique Radon measure on
$\mathcal{E}(G)$ such that if $f\in C_{c}^{\infty}(G)$ then
\[
\int_{\mathcal{E}(G)}\mathrm{tr}(\pi_{\omega}(f))d\mu(\omega)=f(e)
\]
with $e$ the identity element of $G$. Define $\tilde{f}(g)=f(g^{-1})$. Then
\[
\mathrm{tr}(\pi_{\omega^{\ast}}(f))=\mathrm{tr}(\pi_{\omega}(\tilde{f})).
\]
Clearly, $\tilde{f}(e)=f(e)$. Hence we have
\[
f(e)=\int_{\mathcal{E}(G)}\mathrm{tr}(\pi_{\omega}(\tilde{f}))d\mu
(\omega)=\int_{\mathcal{E}(G)}\mathrm{tr}(\pi_{\omega^{\ast}}(f))d\mu
(\omega)=\int_{\mathcal{E}(G)}\mathrm{tr}(\pi_{\omega}(f))d\mu(\omega^{\ast}).
\]

\end{proof}

\end{document}